\documentclass{article}
\usepackage{zformat}
\numberwithin{equation}{section}
\usepackage{authblk}

\title{The action of component groups on irreducible components of Springer fibers}
\usepackage{microtype}
\begin{document}
\author{Do Kien Hoang }
%\affil{Yale University, New Haven, 06520 CT, USA, Partially supported by NSF grant DMS-2001139\\ Email \href{mailto:dokien.hoang@yale.edu}{dokien.hoang@yale.edu}}
\date{}
\maketitle
\begin{abstract}
Let $G$ be a simple Lie group and $e \in \fg$ a nilpotent element. Let $A_e := Z_G(e)/Z_G(e)^{\circ}$ denote the component group of its centralizer, and let $\Irr(\spr)$ denote the set of irreducible components of the Springer fiber over $e$. The group $A_e$ acts naturally on $\Irr(\spr)$. We classify the stabilizers in this action when $\fg$ is of classical type. Our approach combines combinatorial properties of signed domino tableaux with the geometry of a torus-fixed variety to describe the stabilizers in terms of noncrossing partitions. For exceptional $\fg$, we give an explicit description of $\Irr(\spr)$ as an $A_e$-set. These results imply and generalize a conjecture of Lusztig and Sommers. They also give evidence for a connection (first proposed by Lusztig) between Springer fibers, Kazhdan–Lusztig cells, and the finite asymptotic Hecke algebra.
\end{abstract}

\section{Introduction}
Let $G$ be a simple algebraic group. Let $\fg$ be the Lie algebra of $G$. Write $\pi: \tilde{\fg}\rightarrow \fg$ for the Grothendieck-Springer resolution of $\fg$. Let $e$ be a nilpotent element of $\fg$. Let $\spr:= \pi^{-1}(e)$ be the Springer fiber over $e$. The variety $\spr$ is connected but not irreducible in general. We write $\Irr(\spr)$ for the set of its irreducible components. 

Let $Z_G(e)$ be the centralizer of $e$ in $G$, and let $A_e:= Z_G(e)/Z_G(e)^{o}$ be its component group. It is known that $A_e$ takes the form $(\cyclic{2})^{\oplus k}$ for some $k$ when $\fg$ is of type B, C, or D (\cite[Section 6]{collingwood1993nilpotent}). For exceptional $\fg$, we assume that $G$ is of adjoint type. Then $A_e$ is a symmetric group $S_i$ for $i\leqslant 5$. The group $Z_G(e)$ naturally acts on $\spr$, so $A_e$ acts on $\Irr(\spr)$. This action of $A_e$ on $\Irr(\spr)$ gives the space $H^{top}(\spr)$ the structure of an $A_e$-module, here we use $H^{top}$ for the top cohomology group with complex coefficient. This module structure plays an important role in the Springer correspondence and can be described explicitly (see, e.g., \cite[Section 13.3]{carter1993finite}). However, the structure of $\Irr(\spr)$ as an $A_e$-set may not be fully recovered from the $A_e$-module $H^{top}(\spr)$ in general. We aim to obtain more information about $\Irr(\spr)$, the details are as follows.

We first consider $\fg$ of classical type. In type A, the group $A_e$ always acts trivially on $H^{top}(\spr)$ (see, e.g., \cite[Section 3.6]{CG}), so we do not consider $\fg$ of type A in this paper. For type B,C, and D, Van Leeuwen, in his thesis \cite{VanL}, gives a parameterization of the set $\Irr(\spr)$ in terms of signed domino tableaux and describes how one can read the action of $A_e$ on these tableaux. This construction is inductive. The algorithm is rather involved, so several questions remain about this action.

First, we want to know which subgroups $A\subset A_e$ appear as stabilizers. Write $\Stab(e)$ for the set of these subgroups. Second, given $A\in \Stab(e)$, we want to know how many orbits of the form $A_e/A$ appear in $\Irr(\spr)$. These questions also make sense for $\fg$ of exceptional types. A slight change is that $A_e$ may not be commutative, so we consider elements of $\Stab(e)$ as conjugacy classes of subgroups in $A_e$.

This paper provides explicit answers to the first question for classical $\fg$. We parameterize the set $\Stab(e)$ by certain types of noncrossing partitions. These partitions have been studied extensively in enumerative and algebraic combinatorics (see, e.g., \cite{Simion2000}). The result for $\fg$ of type C is as follows. Write $\lambda$ for the partition of $e$. Assume that the distinct even parts of $\lambda$ are $2x_1>...>2x_\ell$ and that the multiplicity of $2x_i$ is $r_i$. Let $B_\lambda$ be the set $\{0,2x_1,...,2x_\ell\}$. 
\begin{thm}[\cref{main thm}]
    There is a bijection between $\Stab(e)$ and the set of noncrossing partitions $\cL=\{L_0,...L_k\}$ of $B_\lambda$ that satisfy $0\in L_0$ and that the numbers $r(L_i)= \sum_{2x_{j}\in L_i}r_j$ are even for $1\leqslant i\leqslant k$.
\end{thm}
For exceptional $\fg$, the first question was answered in \cite{DLP} for distinguished $e$ and in \cite{Sommers2006} for general $e$. For the second question, we propose a conjecture that offers a partial answer in terms of Kazhdan–Lusztig cell theory. This conjecture applies to both classical and exceptional types. To state it, we first recall the relation between KL cells and nilpotent orbits.

The two-sided cells in $W$ are in bijection with the special nilpotent orbits in $\fg$ (\cite{Lusztig1979}). This distinguished class of nilpotent orbits plays a pivotal role in several classical problems in representation theory, including the classification of primitive ideals (\cite{Barbasch1982}) and of irreducible representations of reductive groups over finite fields (\cite{Lusztig1984}). For classical $\mathfrak{g}$, a nilpotent orbit is special if and only if its corresponding partition $\lambda$ has the property that its transpose $\lambda^\intercal$ also corresponds to a nilpotent orbit of $\mathfrak{g}$.

Let $e \in \fg$ be a special nilpotent element, and let $c_e \subset W$ be the corresponding two-sided cell. Lusztig has assigned to the two-sided cell $c_e \subset W$ a quotient $\bar{A}(e)$ of $A_e$, known as Lusztig’s canonical quotient. This quotient plays a pivotal role in describing the constructible representations of $W$, also known as left cell representations. The left cell representations associated with $c_e$ and are classified by Lusztig using certain subgroups $H_\sigma \subset \bar{A}(e)$ (see \cite{Lusztig1982} and \cite{leading}).

Consider the $\bar{A}(e)$-set 
\[
Y= \bigsqcup (\bar{A}(e)/H_\sigma),
\]
where $\sigma$ runs over all left cells in $c_e$. We state the following conjecture.

\begin{conj} \label{main conj}
    Let $e \in \fg$ be a special nilpotent element. Let $K_e$ denote the kernel of the natural map $A_e \twoheadrightarrow \bar{A}(e)$. Then $\Irr(\spr)/K_e$ is isomorphic to $Y$ as a $\bar{A}(e)$-set.
\end{conj}

\subsection{Corollaries}
A corollary of \cref{main conj} is that when the Lusztig quotient $\bar{A}(e)$ coincides with $A_e$, the $A_e$-set $\Irr(\spr)$ is identified with $Y$. On the other hand, the $\bar{A}(e)$-set $Y$ can be explicitly computed using an algorithm in \cite[Section 6.9]{Losev_2014}. In loc.cit., this set $Y$ plays an important role in the classification of finite-dimensional irreducible modules with integral regular central character of the finite W-algebra attached to $e$. Hence, we expect \cref{main conj} to be the combinatorial shadow of a deeper relation between simple modules of finite W-algebras and the irreducible components of Springer fibers. 

Another corollary of \cref{main conj} is a realization of the finite asymptotic Hecke category $\cJ_e$ from \cite{Lusztig1997}. In loc.cit., the category $\cJ_e$ is defined with respect to the two-sided cell $c_e$ so that $\cJ_e$ categorifies a block $J_e$ of the finite asymptotic Hecke algebra. There is a monoidal equivalence between $\cJ_e$ and the category of $\bar{A}(e)$-equivariant sheaves of vector spaces on $Y\times Y$, $\Coh^{\bar{A}(e)}(Y\times Y)$ (first conjectured in \cite{leading}, and proved in \cite{Tensor3}). The tensor product of $\Coh^{\bar{A}(e)}(Y\times Y)$ is given by convolution. Therefore, \cref{main conj} would mean that we can replace $Y$ by $\Irr(\spr)/K_e$ to obtain a realization of $\cJ_e$ in terms of the Springer fiber $\spr$. Passing to K-groups, we see that $K_0^{\bar{A}(e)}((\Irr(\spr)/K_e)\times (\Irr(\spr)/K_e))$ is isomorphic to $J_e$. On the other hand, the corresponding block $J_e^{aff}$ of the affine asymptotic Hecke algebra admits a geometric realization as $K_0^{Z_e}(\fix\times \fix)$ in \cite{bezrukavnikov2023geometric} where $Z_e$ is the reductive part of $Z_G(e)$ and $\fix$ is a certain smooth subvariety of $\spr$. One of our further goals is to establish connections between these realizations of $J_e$ and $J_e^{aff}$. %\footnote{The idea of relating the finite Weyl group and the connected components of $\fix\times \fix$ was originally suggested by Roman Bezrukavnikov.}

In this paper, our second main result is evidence of \cref{main conj} for classical $\fg$. This result is also a stronger form of a conjecture by Lusztig and Sommers from \cite{lusztig2024constructible}.
\begin{thm}[\cref{evidence}]
    Assume that $\fg$ is of classical type and $e\in \fg$ is a special nilpotent element. Consider the quotient map $p: A_e\rightarrow \bar{A}(e)$. The set $\{p(A), A\in \Stab(e)\}$ consists precisely of the subgroups $H_\sigma\subset \bar{A}(e)$ attached to the left cells in $c_e$.
\end{thm}
\subsection{Contents}
When $\fg$ is of exceptional type, we give an explicit description of the $A_e$-action on $\Irr(\spr)$ and verify that \cref{main conj} holds. These results are presented in Section 5.

The rest of the paper focuses on the case where $\fg$ is of classical type. Section 2 revisits the construction of Van Leeuwen \cite{VanL}, who described $\Irr(\spr)$ in terms of signed domino tableaux. Using his notion of clusters (\cite[Section 3.3]{VanL}), we formulate a necessary condition for a subgroup $A \subset A_e$ to arise as a stabilizer.

In Section 3, we prove that this condition is also sufficient. More precisely, for every subgroup $A \in \Stab(e)$, we construct an irreducible component of $\spr$ whose stabilizer is exactly $A$. The proof relies on geometric properties of a smooth torus-fixed subvariety $\fix \subset \spr$, which appears in several works \cite{DLP, lusztig2021discretization, bezrukavnikov2023geometric}, and has been studied in detail in \cite{hoang2024geometryfixedpointsloci}. We also explain a key step, inspired by \cite[Section 3]{DLP}, that relates the connected components of $\fix$ to the irreducible components of $\spr$.

Section 4 establishes \cref{evidence}, which identifies the image of $\Stab(e)$ under the canonical map $A_e \to \bar{A}(e)$ with the collection of Lusztig’s subgroups $H_\sigma$ arising from constructible representations. We also briefly outline an approach to \cref{main conj} for classical $\fg$.
\begin{comment}
    \subsection{Acknowledgements.} I would like to thank George Lusztig for introducing me to the conjecture in \cite{lusztig2024constructible}. Another source of motivation for this paper was a series of stimulating discussions with Roman Bezrukavnikov and Vasily Krylov.

I am also grateful to Eric Sommers for many fruitful discussions and insightful comments on various aspects of the text. In particular, his code and explanations related to \cite{DLP} were crucial to my understanding of the exceptional case. I also thank Ivan Losev for countless valuable pieces of advice. This work was partially supported by the NSF under grant DMS-2001139.
\end{comment}

\section{Signed domino tableaux and a combinatorial description of $\Stab(e)$}

In this paper, for classical groups, we explain the details for $G = \mathrm{Sp}_{2n}$. The cases where $G$ is of type B or D are similar; we state only general definitions and main results for these cases.

We recall the parameterization of nilpotent orbits in $\fsp_{2n}$. From \cite[Theorem 5.1.3]{collingwood1993nilpotent}, the nilpotent orbits in $\fsp_{2n}$ are in bijection with the following set of partitions:
$$
\cP_C(2n) := \{\text{partitions } \lambda \text{ of } 2n \text{ in which odd parts have even multiplicities}\}.
$$
We say that a partition is of type C if it belongs to $\cP_C(2m)$ for some $m$.

Next, we describe the group $A_e$ in terms of $\lambda$, the partition associated to $e$. Consider $\lambda \in \cP_C(2n)$, and let $\ell$ be the number of distinct even parts in $\lambda$. Assume that they are $2x_1 > \dots > 2x_\ell$ with multiplicities $r_1, \dots, r_\ell$ in $\lambda$. From \cite[Section 6.1]{collingwood1993nilpotent}, the group $A_e$ is isomorphic to $(\cyclic{2})^{\oplus \ell}$. There is a natural set of generators $t_1, \dots, t_\ell$ of $A_e$ corresponding to the parts $2x_1, \dots, 2x_\ell$ of $\lambda$.

The main result of this section is a bijection between the set $\Stab(e)$ and the set of \textit{even noncrossing partitions} of $\{0, 2x_1, \dots, 2x_\ell\}$. In Section 2.1, we introduce the notion of even noncrossing partitions and state the results in \cref{main thm} and \cref{main BD}.

\subsection{Noncrossing partitions and the main results}

Let $B_\lambda$ be the set $\{0, 2x_1, \dots, 2x_\ell\}$. We recall the definition of noncrossing partitions of $B_\lambda$ (see, e.g., \cite{Simion2000}).

\begin{defin}[NCP]\label{NCP}
    Consider a partition of the set $B_\lambda = \{0, 2x_1, \dots, 2x_\ell\}$ into $k+1$ nonempty subsets $L_0, L_1, \dots, L_k$. For $0 \leqslant i < j \leqslant k$, consider $a \leqslant b$ in $L_i$ and $c \leqslant d$ in $L_j$. 

    For the two intervals $[a, b]$ and $[c, d]$, we impose the following conditions:
    \begin{enumerate}
        \item The two intervals do not intersect; or
        \item One interval is contained in the other.
    \end{enumerate}

    If this condition is satisfied for all pairs $0 \leqslant i < j \leqslant k$ and for all $a, b \in L_i$ and $c, d \in L_j$, we call the partition $L_0, L_1, \dots, L_k$ a \textit{noncrossing partition} (abbreviated NCP) of $B_\lambda$.
\end{defin}

Next, we explain an illustrative way to think about the NCPs. Let $L_0,L_1, \dots, L_k$ be a partition of $\{0,2x_1,\dots,2x_\ell\}$. Pick $k+1$ different colors with respect to $k+1$ sets $L_i$. Consider the points $\{0,2x_1,\dots,2x_\ell\}$ placed on the $x$-axis in $\mathbb{R}^{2}$. Connect two points $a< b$ in $L_i$ if $[a,b]\cap L_i= \{a,b\}$ by an arc above the $x$-axis. For each $L_i$, draw an arc above the $x$-axis connecting two points $a < b$ in $L_i$ if $[a, b] \cap L_i = \{a, b\}$. Color the points in $L_i$ and the arcs between them using the color assigned to $L_i$. Then $L_0, L_1, \dots, L_k$ forms an NCP if and only if no two arcs of different colors intersect in the diagram. We refer to this as an \textit{NCP diagram}.

\begin{rem}\label{end point}\leavevmode
\begin{enumerate}
    \item The NCPs of $\{0,2x_1,\dots,2x_\ell\}$ are in bijection with the NCPs of $\{0,1,\dots,\ell\}$. 
    \item An NCP $L_0,L_1,\dots,L_k$ of $B_\lambda= \{0,2x_1,\dots,2x_\ell\}$ is uniquely determined by the numbers $min(L_i)$, $max(L_i)$, $0\leqslant i\leqslant \ell$. In particular, assume that we are given $k+1$ intervals 
    $I_i = [\min(L_i), \max(L_i)]$ 
    that satisfy the conditions in \cref{NCP}.
    
    For each $i$, let $D_i \subset [0, k]$ be the set of indices $j$ such that $I_j \subsetneq I_i$.
    Then
    $$
    L_i = I_i \setminus \left( \bigcup_{j \in D_i} I_j \right).
    $$
    
    \item Similarly, assume that we know $0\in L_0$, then $I_0$ is not contained in $I_i$ for $1\leqslant i\leqslant k$. In this case, the NCP is determined by the $k$ intervals $I_1, \dots, I_k$. We obtain $L_1, \dots, L_k$ as above and we get $L_0$ as the complement of $\bigcup_{i=1}^{k} L_i$ in $B_\lambda$.
\end{enumerate}
\end{rem}

\begin{exa}
    In what follows, we present several NCP diagrams for the cases $\ell = 4$ and $\ell = 5$.
    $$
    \begin{tikzpicture}
        \draw[red] (-10,0) .. controls (-8,1) .. (-6,0); \filldraw[red] (-10,0) circle (2pt); \filldraw[red] (-6,0) circle (2pt);
\draw[blue] (-9,0) .. controls (-8,0.5) .. (-7,0); \filldraw[blue] (-9,0) circle (2pt); \filldraw[blue] (-7,0) circle (2pt);
\filldraw[yellow] (-8,0) circle (2pt);

\draw[blue] (-1,0) .. controls (-1.5,0.25) .. (-2,0); \filldraw[blue] (-1,0) circle (2pt); \filldraw[blue] (-2,0) circle (2pt); 
\draw[blue] (-2,0) .. controls (-2.5,0.25) .. (-3,0); \filldraw[blue] (-3,0) circle (2pt);
\draw[blue] (-3,0) .. controls (-3.5,0.25) .. (-4,0); \filldraw[blue] (-4,0) circle (2pt);
\filldraw[red] (0,0) circle (2pt);
    \end{tikzpicture}
    $$
    
    $$
\begin{tikzpicture}

\draw[red] (-5.5,0) .. controls (-5,0.25) .. (-4.5,0); \filldraw[red] (-4.5,0) circle (2pt); \filldraw[red] (-5.5,0) circle (2pt);
\draw[green] (-2.5,0) .. controls (-3,0.25) .. (-3.5,0); \filldraw[green] (-3.5,0) circle (2pt); \filldraw[green] (-2.5,0) circle (2pt);
\draw[red] (-1.5,0) .. controls (-1,0.25) .. (-0.5,0); \filldraw[red] (-1.5,0) circle (2pt); \filldraw[red] (-0.5,0) circle (2pt);
\draw[red] (-4.5,0) .. controls (-3,1) .. (-1.5,0);

\draw[red] (0,0) .. controls (2.5,1) .. (5,0); \filldraw[red] (0,0) circle (2pt); \filldraw[red] (5,0) circle (2pt);
\draw[blue] (1,0) .. controls (1.5,0.25) .. (2,0); \filldraw[blue] (1,0) circle (2pt); \filldraw[blue] (2,0) circle (2pt); 
\draw[blue] (2,0) .. controls (2.5,0.25) .. (3,0); \filldraw[blue] (3,0) circle (2pt);
\draw[blue] (3,0) .. controls (3.5,0.25) .. (4,0); \filldraw[blue] (4,0) circle (2pt);

\end{tikzpicture}
$$
\end{exa}
Next, we define the even noncrossing partitions (even NCPs) with respect to a type C partition $\lambda$.
Recall that we write $r_i$ for the multiplicity of $2x_i$ in $\lambda$, for $1 \leqslant i \leqslant \ell$.

\begin{defin}\label{even NCP}
A partition $\cL = \{L_0, \dots, L_k\}$ of $B_\lambda$ is called an \textit{even noncrossing partition} (even NCP) if the following conditions are satisfied:
\begin{enumerate}
    \item $\cL$ is an NCP of $B_\lambda$, and $0 \in L_0$.
    \item For each $1 \leqslant i \leqslant k$, the number
    $$
    r(L_i) := \sum_{2x_j \in L_i} r_j
    $$
    is even.
\end{enumerate}
\end{defin}

Recall that $A_e \cong (\cyclic{2})^{\oplus \ell}$ has a set of generators $t_1, \dots, t_\ell$.
Given an NCP $\cL = \{L_0, \dots, L_k\}$ of $B_\lambda$, we define the corresponding subgroup $A_{\cL} \subset A_e$ to be generated by:
\begin{itemize}
    \item $t_i t_j$ for $2x_i, 2x_j \in L_l$, where $1 \leqslant l \leqslant k$;
    \item $t_i$ for $2x_i \in L_0$.
\end{itemize}

Our main theorem for type C is as follows.

\begin{thm} \label{main thm}
    Let $S(e)$ be the set consisting of the groups $A_\cL$, where $\cL$ ranges over the even NCPs of $B_\lambda$. Then $S(e)$ coincides with $\Stab(e)$.

\end{thm}
The proof of this theorem is given in Sections 2.3 and 3.
%\cref{main combi} has shown that $\Stab(e)\subset S(e)$, we will prove $S(e)\subset \Stab(e)$ in the next section. 

\begin{exa} \label{example stab}
    Consider $\lambda = (100^3, 38^3, 16^2)$.
    Then $A_e$ is generated by $t_1, t_2, t_3$, and the set $B_\lambda$ is $\{0, 100, 38, 16\}$.
    There are five even NCPs in this case. We list the elements of the set $S(e)$ together with their corresponding even NCPs.
    \begin{enumerate}
        \item $A = \langle t_1 t_2, t_3 \rangle$ corresponds to $\{100, 38\}, \{0, 16\}$.
        \item $A = \langle t_1 t_2 \rangle$ corresponds to $\{100, 38\}, \{16\}, \{0\}$.
        \item $A = \langle t_1, t_2 \rangle$ corresponds to $\{100, 38, 0\}, \{16\}$.
        \item $A = \langle t_1, t_2, t_3 \rangle$ corresponds to $\{100, 38, 16, 0\}$.
        \item $A = \langle t_1 t_2, t_2 t_3 \rangle$ corresponds to $\{100, 38, 16\}, \{0\}$.
    \end{enumerate}
\end{exa}

We state a similar result for $\fg = \fso_{2n}$ or $\fso_{2n+1}$. 
Consider $e \in \fg$ with the associated partition $\lambda$.
Let $2x_1 + 1, \dots, 2x_\ell + 1$ be the distinct odd parts of $\lambda$, and write $r_i$ for the multiplicities of $2x_i + 1$ in $\lambda$.
Set $B_\lambda = \{0, 2x_1 + 1, \dots, 2x_\ell + 1\}$, and define the even NCPs $\cL$ of $B_\lambda$ as in \cref{even NCP}.

The group $A_e$ is trivial if $\ell = 0$. 
When $G = O_{2n}$ (or $O_{2n+1}$) and $\ell \geqslant 1$, the group $A_e$ is isomorphic to $(\cyclic{2})^{\ell}$ with a set of generators $t_1, \dots, t_\ell$ \cite[Section 6]{collingwood1993nilpotent}. 
For $G = SO_{2n}$ (or $SO_{2n+1}$), we instead have $A_e \cong (\cyclic{2})^{\ell - 1}$, with generators $t_i t_j$ for $1 \leqslant i < j \leqslant \ell$. 
In either case, the process of obtaining a subgroup $A_{\cL} \subset A_e$ from an even NCP of $B_\lambda$ makes sense.

In the case of $SO$, the generators of $A_\cL$ that come from $L_0$ are now taken to be $t_i t_j$ for $i, j \in L_0$.

\begin{thm}\label{main BD}
    Let $S(e)$ be the set consisting of the groups $A_\cL$, where $\cL$ ranges over the even NCPs of $B_\lambda$.
    Then $S(e)$ coincides with $\Stab(e)$.
\end{thm}

In the remainder of this section, we explain the containment $\Stab(e) \subset S(e)$ in the case where $e \in \fsp_{2n}$. The key ingredient is a parameterization of the set $\Irr(\spr)$ using signed domino tableaux. The cases of $\fso_{2n}$ and $\fso_{2n+1}$ can be treated in the same way.

\subsection{Generalities about signed domino tableaux}
\subsubsection{Standard domino tableaux and admissible domino tableaux}
We recall the definitions of standard domino tableaux and admissible domino tableaux. The following exposition follows \cite{VanL} and \cite{Pietraho2004}. 

Consider a partition $\lambda$ of a positive integer $N$. We have the corresponding Young diagram $Y_\lambda$ consisting of $N$ boxes. We call $\lambda$ the \textit{shape} of this Young diagram. We will use the notation $(k,l)$ for the box in the $k$-th row and the $l$-th column (from the top and the left). A horizontal domino then consists of two boxes $(i,j)$ and $(i, j+1)$; a vertical domino consists of two boxes $(i,j)$ and $(i+1,j)$ for some $i,j$. 
\begin{defin}\label{standard domino tableaux}
    A \text{standard domino tableau} is a partition of $Y_\lambda$ into dominoes of sizes $1\times 2$ or $2\times 1$ and possibly a single box that satisfies the following.
    \begin{enumerate}
        \item If $N$ is odd, the box $(1,1)$ forms a singleton in the partition. In particular, there are always $n= \floor*{\frac{N}{2}}$ dominoes. 
        \item These $n$ dominoes are labeled with $n$ distinct numbers $1,\dots,n$. Write $f(i,j)$ for the label of the domino containing the box $(i,j)$. If $N$ is odd, let $f(1,1)$ be $0$.
        \item $f(i,j)\leqslant f(i+1,j)$ and $f(i,j)\leqslant f(i,j+1)$ for $i,j\geq 1$. 
    \end{enumerate}
    The set of standard domino tableaux of shape $\lambda$ is denoted by $SDT(\lambda)$.
\end{defin}
\begin{exa}\label{standardex}
    In the following, we list all the standard domino tableaux of shape $(3,3)$.
    $$
\raisebox{3ex}{$T_1=$ \;}
\begin{tableau}
:^1 ^2 ^3\\
:;\\
\end{tableau}
\hspace{1in}
\raisebox{3ex}{$T_2=$ \;}
\begin{tableau}
:^1 >2\\
:; >3\\
\end{tableau}
\hspace{1in}
\raisebox{3ex}{$T_3=$ \;}
\begin{tableau}
: >1 ^3\\
: >2\\
\end{tableau}
$$
\end{exa}

We proceed to the definitions of admissible domino tableaux of type $X$, where $X = B, C$, or $D$. Let $\lambda$ be a partition of $2n$ or $2n+1$. 
Recall from \cite{collingwood1993nilpotent} that nilpotent orbits in classical types are parameterized by their partitions. 
In type $D$, a partition consisting only of even parts corresponds to two orbits. 
The Springer fibers of these two orbits are isomorphic, so we will not distinguish between them.

\begin{defin} \label{admissible domino tableaux}
    Consider a standard domino tableau $T$. 
    For $1 \leqslant i \leqslant n$, let $T^i$ denote the domino tableau consisting of the dominoes in $T$ with labels less than $i+1$. 
    We say that $T$ is \textit{admissible of type $X$} if the shape of $T^i$ is a partition corresponding to a nilpotent orbit of type $X$, for each $i$.
\end{defin}

From \cite[Section 3.2]{VanL}, an admissible tableau $T$ consists of three types of dominoes, defined as follows.

\begin{defin}\label{type of domino}
    Suppose $T$ is admissible of type $C$ (respectively, $B$ or $D$).
    \begin{enumerate}
        \item Type $(N)$: a horizontal domino, where the left box lies in an odd (respectively, even) column.
        \item Type $(I+)$: a vertical domino that lies in an even (respectively, odd) column.
        \item Type $(I-)$: a vertical domino that lies in an odd (respectively, even) column.
    \end{enumerate}
\end{defin}

In the remainder of the paper, we work with $\fsp_{2n}$.

\begin{exa}\label{admissible}
    Consider the domino tableaux in \cref{standardex}. 
    The tableau $T_2$ is not admissible because $T_2^2$ has the shape $(3,1) \notin \cP_C(4)$. Alternatively, one can observe that the dominoes labeled $2$ and $3$ in $T_2$ are not of the three types defined above.

    On the other hand, $T_1$ and $T_3$ are admissible. The tableau $T_1$ consists of one $(I+)$ domino and two $(I-)$ dominoes, while $T_3$ consists of two $(N)$ dominoes and one $(I-)$ domino.
\end{exa}

Now consider $\lambda \in \cP_C(2n)$, and let $e$ be a nilpotent element of $\fsp_{2n}$ with associated partition $\lambda$. Let $\sum DT(\lambda)$ denote the set of admissible domino tableaux of shape $\lambda$, together with a sign $+$ or $-$ assigned to each domino of type $(I+)$.

The set $\Irr(\spr)$ is parameterized by a quotient of $\sum DT(\lambda)$ by an equivalence relation. 
To explain this equivalence relation, we recall the definition of clusters in the next section.

\subsubsection{Clusters}
Consider an admissible domino tableau $T$. Following \cite[Section 3.3]{VanL}, we inductively define the set of clusters $\cC_T$ and a map $b_T: B_\lambda \rightarrow \cC_T$.

The clusters form a partition of $T$. When $T$ consists of a single domino, there is a single cluster, and $b_T$ sends $0$ to this cluster.

Now consider the case where $T$ has $n$ dominoes. Recall that we write $T^{n-1}$ for the admissible tableau consisting of the dominoes labeled $1, 2, \dots, n-1$ in $T$. Let $\lambda'$ be the shape of $T^{n-1}$. Assuming that the set $\cC_{T^{n-1}}$ and the map $b_{T^{n-1}}$ are already defined, we define $\cC_T$ and $b_T$ as follows.

\begin{defin}\label{clusters def}
    Let $d$ be the $n$-th domino. There are three cases, depending on the type of $d$:
    \begin{enumerate}
        \item[(N)] $d$ lies in columns $2i-1$ and $2i$ for some $i \geqslant 1$. Then $2i \in B_\lambda$. We merge the cluster $b_{T^{n-1}}(2i-2)$, the domino $d$, and the cluster $b_{T^{n-1}}(2i)$ into a new cluster of $T$. 
        Note that the cluster $b_{T^{n-1}}(2i)$ only exists if $2i$ is a part of $\lambda'$. 
        All other clusters of $T^{n-1}$ remain unchanged.
        %and their preimages under $b_T$ are inherited from $b_{T^{n-1}}^{-1}$. The remaining numbers in $B_\lambda$ are sent to the new cluster.
        
        \item[(I+)] $d$ lies in column $2i$ for some $i \geqslant 1$. Then $2i \in B_\lambda$. 
        We merge $d$ and the cluster $b_{T^{n-1}}(2i)$ into a new cluster of $T$. 
        In particular, if $2i$ is not a part of $\lambda'$, the new cluster consists only of the domino $d$. All other clusters of $T^{n-1}$ remain unchanged.
        %The remaining steps are the same as in the $(N)$ case.

        \item[(I-)] $d$ lies in column $2i+1$ for some $i \geqslant 1$. 
        We merge $d$ and the cluster $b_{T^{n-1}}(2i)$. All other clusters of $T^{n-1}$ remain unchanged.
        %The remaining steps are the same as in the $(N)$ case.
    \end{enumerate}
    For the unchanged clusters of $T^{n-1}$, their preimages under $b_T$ are inherited from $b_{T^{n-1}}^{-1}$. Finally, we send the remaining unassigned elements of $B_\lambda$ to the new cluster.
\end{defin}

The clusters in the image of $b_T(B_\lambda)$ are called \textit{open}, and the remaining clusters are called \textit{closed}. 
With this construction, any cluster $\cC \neq b_T(0)$ always contains at least one domino of type $(I+)$ \cite[Section 3.3]{VanL}.

Given an element of $\sum DT(\lambda)$, we define the sign of a cluster $\cC \neq b_T(0)$ as the product of the signs of the dominoes of type $(I+)$ in $\cC$. 
The sign of the cluster $b_T(0)$ is, by definition, taken to be $+$.

For $T' \in \sum DT(\lambda)$, let its \textit{underlying tableau} be the domino tableau obtained from $T'$ by forgetting the signs $\pm$. 
Consider $T_1, T_2 \in \sum DT(\lambda)$. 
We say $T_1 \sim_{cl} T_2$ (respectively, $T_1 \sim_{op,cl} T_2$) if the following conditions are satisfied:
\begin{enumerate}
    \item $T_1$ and $T_2$ have the same underlying tableau;
    \item The signs of all closed (respectively, open and closed) clusters in $T_1$ and $T_2$ are equal.
\end{enumerate}

Let $\sum DT_{cl}(\lambda)$ and $\sum DT_{op,cl}(\lambda)$ denote the quotients $\sum DT(\lambda)/\sim_{cl}$ and $\sum DT(\lambda)/\sim_{op,cl}$, respectively. 
The irreducible components of $\spr$ are parametrized by the elements of $\sum DT_{op,cl}(\lambda)$ \cite[Proposition 3.4.1]{VanL}. 
Two irreducible components of $\spr$ lie in the same $A_e$-orbit if and only if the corresponding signed domino tableaux have the same signs on the closed clusters.

In other words, the forgetful map 
$$
\sum DT_{op,cl}(\lambda) \rightarrow \sum DT_{cl}(\lambda)
$$ 
sends an irreducible component of $\spr$ to its $A_e$-orbit.

\begin{exa} \label{22 Springer}
Consider $e\in \fsp_4$ with the associated partition $(2,2)$. The singed domino tableaux that parameterize the irreducible components of $\spr$ are as follows. 
    $$
\raisebox{3ex}{$T_1=$ \;}
\begin{tableau}
:^1 ^{2\atop +}\\
:;\\
\end{tableau}
\hspace{1in}
\raisebox{3ex}{$T_2=$ \;}
\begin{tableau}
:^1 ^{2\atop -}\\
:;\\
\end{tableau}
\hspace{1in}
\raisebox{3ex}{$T_3=$ \;}
\begin{tableau}
:>1\\
:>2\\
\end{tableau}
$$
\newline
For $T_1$ and $T_2$, there are two clusters; each consists of a single domino. The map $b_{T_1}= b_{T_2}$ sends $0$ and $2$ to the clusters $\{1\}$ and $\{2\}$, respectively. In $T_3$, there is only one cluster and $b_{T_3}$ sends both $0$ and $2$ to this cluster. The action of $A_e$ fixes $T_3$, while permutes $T_1$ and $T_2$. In this example, all the clusters mentioned are open.
\end{exa}
Next, we present examples with the appearance of closed clusters.
\begin{exa} \label{44 Springer} %define b_T first maybe
Let $\lambda= (4,4,2,2)$. Consider
$$\raisebox{3ex}{$T=$ \;}
    \begin{tableau}
:^1^2^3^4\\
:;\\
:^5^6\\
\end{tableau}$$ 
The clusters are $\{1,5\}_0$, $\{2,3\}$, $\{4\}_4$, and $\{6\}_2$. Here, the subscript of a cluster is the column attached to it by $b_T$. The open clusters are $\{1,5\}_0$, $\{4\}_4$, and $\{6\}_2$. Hence, we have correspondingly $8$ elements of $\sum DT_{op, cl}(\lambda)$ with the underlying domino tableau $T$. The elements of $\sum DT_{op, cl}(\lambda)$ with underlying tableau $T$ are 
$$ \quad
\begin{tableau}
:^1^{2\atop +}^3^{4\atop \pm}\\
:;\\
:^5^{6\atop \pm}\\
\end{tableau} \quad \text{and} \quad
\begin{tableau}
:^1^{2\atop -}^3^{4\atop \pm}\\
:;\\
:^5^{6\atop \pm}\\
\end{tableau}
$$
\newline
These tableaux parameterize $8$ irreducible components that belong to two diffent $A_e$-orbits (corresponding to two different signs of the closed cluster \{2,3\}). 

Next, we consider
$$\raisebox{3ex}{$T=$ \;}
    \begin{tableau}
:^1^3>5\\
:; ; >6\\
:^2^4\\
\end{tableau}$$ 
\newline
Here, the clusters are $\{1,2\}_{0}$ and $\{3,4,5,6\}_{2,4}$. 
The cluster $\{3,4,5,6\}_{2,4}$ contains two $(I+)$ dominoes labeled $3$ and $4$. 
To represent the classes of $\sum DT_{op,cl}(\lambda)$ with underlying tableau $T$, we may fix the sign of the domino labeled $4$. 
Hence, we obtain the following two classes.

$$ \quad
\begin{tableau}
:^1^{3\atop \pm}>5\\
:; ; >6\\
:^2^{4\atop +}\\
\end{tableau}
$$
\newline
Consequently, the $A_e$-orbit of $\Irr(\spr)$ that corresponds to $T$ has two irreducible components. 
\end{exa}

\subsection{The action of $A_e$}
In this section, we explain how $A_e$ acts on the set $\sum DT_{op, cl}(\lambda)$. The main result of this section is \cref{main combi} where we prove the containment $\Stab(e)\subset S(e)$ (see \cref{main thm}). 

Consider $\lambda\in \cP_C(2n)$. Consider an admissible domino tableau $T$ of shape $\lambda$, write $\cO\cC_T$ for the set of open clusters in $T$. From \cref{clusters def}, we have the surjective map $b_T: B_\lambda= \{0,2x_1,\dots,2x_\ell\} \rightarrow \cO\cC_T$.

Let $\sum DT_{op,cl}(T)$ be the subset of $\sum DT_{op,cl}(\lambda)$ with the underlying domino tableau $T$. In \cite[Section 3.3]{VanL}, an action of $A_e$ on $\sum DT_{op,cl}(T)$ is defined as follows. The generator $t_i$ of $A_e$ acts on $\sum DT_{op,cl}(T)$ by changing the sign of the open cluster $b_T(2x_i)$. In particular, if $2x_i$ and $2x_j$ are assigned to the same open cluster by $b_T$, the actions of $t_i$ and $t_j$ on $\sum DT_{op,cl}(T)$ are the same. The action of $t_i$ for $2x_i\in b_T(0)$ is always trivial. More precisely, we have the following corollary.
\begin{cor} \label{easy cor}
    The stabilizer of an element $T^\pm \in \sum DT_{op,cl}(T)$ is the subgroup of $A_e$ that is generated by elements of the form $t_it_j$ for $b_T(i)= b_T(j)$ and $t_i$ for $2x_i\in b_T(0)$.
\end{cor}

From \cite[Proposition 3.4.1]{VanL}, there is an $A_e$-equivariant bijection between $\Irr(\spr)$ and $\sum DT_{op,cl}(\lambda)$. Therefore, the stabilizer set $\Stab(e)$ is determined by all possible partitions of $B_\lambda$ given by the preimages of the map 
$$
b_T: \{0, 2x_1, \dots, 2x_\ell\} \rightarrow \cO\cC_T.
$$ 
The following result relates these partitions to the even noncrossing partitions defined in \cref{even NCP}.

\begin{pro}\label{main combi}Consider an admissible domino tableau $T$ of shape $\lambda$. The partition of $\{0,2x_1,\dots,2x_\ell\}$ into the preimages of the open clusters is an even noncrossing partition (even NCP) of $\{0,2x_1,\dots,2x_\ell\}$. As a result, $\Stab(e)$ is a subset of $S(e)$.
\end{pro}
\begin{proof}
    
We prove the proposition by induction on the number of dominoes in $T$. The proof closely follows the inductive process used to define the clusters in \cref{clusters def}. 

The base case is when $T$ consists of a single domino; in this case, the conditions are clearly satisfied. 
Consider $T$ has $n>1$ dominoes. Let $\lambda'$ be the shape of $T^{n-1}$. Let $\cL'$ and $\cL$ be the partition defined by $b_{T^{n-1}}$ and $b_T$, respectively. Assume that $\cL'$ is an even NCP, we prove $\cL$ has the same property. 

\textbf{Step 1: Verify the noncrossing condition.} We explain how to obtain the NCP diagram of $\cL$ from that of $\cL'$. We consider three cases, depending on the type of the new domino $d$ labeled $n$. 

    \begin{enumerate}
        \item (I+): $d$ lies in the $2i$-th column. If $2i\in \lambda'$, the NCP diagram remains unchanged. If $2i\notin \lambda'$, we add the point $2i$ to the NCP diagram with a new color, representing a new open cluster consisting of a single domino.
        \item (I-): $d$ lies in the $2i+1$-th column. If $2i\in \lambda$, the NCP diagram remains unchanged. If $2i\notin \lambda$, we remove the point $2i$ and concatenate the two arcs that had $2i$ as an endpoint. The following diagram illustrates this operation.
        
        $$
        \begin{tikzpicture}[scale=1.0]

% LEFT SIDE (before)
% Arc 1: red
\draw[red] (-7,0) .. controls (-6.5,0.25) .. (-6,0); 
\filldraw[red] (-7,0) circle (2pt); 
\filldraw[red] (-6,0) circle (2pt);

% Arc 2: green
\draw[green!60!black] (-4.5,0) .. controls (-4,0.25) .. (-3.5,0); 
\filldraw[green!60!black] (-4.5,0) circle (2pt); 
\filldraw[green!60!black] (-3.5,0) circle (2pt);

% Arc 3: red, with label 2i
\draw[red] (-2.5,0) .. controls (-2,0.25) .. (-1.5,0); 
\filldraw[red] (-2.5,0) circle (2pt) node[anchor=north]{\(2i\)}; 
\filldraw[red] (-1.5,0) circle (2pt); 

% Big arc over: red
\draw[red] (-6,0) .. controls (-4,1) .. (-2.5,0); 

% ARROW (center)
\draw[->, thick] (0,0.5) -- (1.5,0.5) node[midway, above]{remove \(2i\), merge arcs};

% RIGHT SIDE (after)
% New big arc: red
\draw[red] (3,0) .. controls (5,1) .. (7,0); 
\filldraw[red] (3,0) circle (2pt); 
\filldraw[red] (7,0) circle (2pt);

% Green arc (unchanged)
\draw[green!60!black] (4,0) .. controls (4.5,0.25) .. (5,0); 
\filldraw[green!60!black] (4,0) circle (2pt); 
\filldraw[green!60!black] (5,0) circle (2pt);

% Short red arc
\draw[red] (2,0) .. controls (2.5,0.25) .. (3,0); 
\filldraw[red] (2,0) circle (2pt); 
% No point at 2.5 = formerly 2i

\end{tikzpicture}
$$
        In the case where $2i$ is the maximum or minimum of the preimage of an open cluster of $T^{n-1}$, we simply remove the point $2i$ and the arc connected to it. 
        \item (N): $d$ lies in two columns $2i-1$ and $2i$ for some $i\geqslant 1$. We consider the following subcases.
        \begin{enumerate}
            \item Suppose \( 2i \notin \lambda' \) and \( 2i - 2 \notin \lambda \). In the NCP diagram of \( b_{T^{n-1}} \), we shift the point labeled \( 2i - 2 \) to the right, replacing it with the point labeled \( 2i \). Since there are no other even numbers between \( 2i - 2 \) and \( 2i \), the resulting diagram still satisfies the noncrossing condition.
            \item Suppose \( 2i \notin \lambda' \) and \( 2i - 2 \in \lambda \). We first add the point \( 2i \) to the NCP diagram of \( b_{T^{n-1}} \), and remove the arc connecting \( 2i - 2 \) to some \( 2x_l \) with \( x_l > i \). We then insert two new arcs: one from \( 2i - 2 \) to \( 2i \), and another from \( 2i \) to \( 2x_l \). (If no such \( x_l \) exists, we simply add the arc between \( 2i - 2 \) and \( 2i \).)

            \item Suppose $2i \in \lambda'$. If $2i$ and $2i - 2$ have the same color, the NCP diagram remains unchanged. If they have different colors, we first add an arc between $2i - 2$ and $2i$, then recolor all the points in $b_{T^{n-1}}^{-1}(b_{T^{n-1}}(2i))$, $b_{T^{n-1}}^{-1}(b_{T^{n-1}}(2i - 2))$, and the associated arcs with a new color. This process corresponds to merging two open clusters into one; that is, we take the union of two sets in an NCP to obtain a new one. The following is an illustration.

            $$
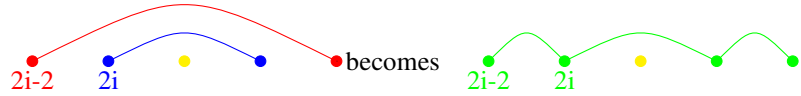
\begin{tikzpicture}[scale=1]

% --- BEFORE (left side) ---
\draw[red] (-8,0) .. controls (-6,1) .. (-4,0); 
\filldraw[red] (-8,0) circle (2pt) node[anchor=north]{$2i{-}2$}; 
\filldraw[red] (-4,0) circle (2pt);

\draw[blue] (-7,0) .. controls (-6,0.5) .. (-5,0); 
\filldraw[blue] (-7,0) circle (2pt) node[anchor=north]{$2i$}; 
\filldraw[blue] (-5,0) circle (2pt);

\filldraw[yellow] (-6,0) circle (2pt);

% --- ARROW ---
\draw[->, thick] (-3.2,0.5) -- (-1.8,0.5) node[midway, above] {\small merge colors};

% --- AFTER (right side) ---

% Green arc: 1st to 2nd
\draw[green!70!black] (0,0) .. controls (0.5,0.5) .. (1,0);
\filldraw[green!70!black] (0,0) circle (2pt) node[anchor=north]{$2i{-}2$};
\filldraw[green!70!black] (1,0) circle (2pt) node[anchor=north]{$2i$};

% Yellow middle point (3rd)
\filldraw[yellow] (2,0) circle (2pt);

% Green arc: 2nd to 4th
\draw[green!70!black] (1,0) .. controls (2,0.5) .. (3,0);
\filldraw[green!70!black] (3,0) circle (2pt);

% Green arc: 4th to 5th
\draw[green!70!black] (3,0) .. controls (3.5,0.5) .. (4,0);
\filldraw[green!70!black] (4,0) circle (2pt);

\end{tikzpicture}
$$

        \end{enumerate}
    
    \end{enumerate}
\end{proof}
\textbf{Step 2: Verify the even condition.} We verify that $\cL$ is an even NCP (see \cref{even NCP}). By definition of $b_T$, we have $0\in L_0$. Consider an even part $2x_j\in \lambda$. Let $r_j$ and $r_j'$ be the multiplicities of $2x_j$ in $\lambda$ and $\lambda'$. Provided that $\cL'$ is even, the claim that $\cL$ satisfies Condition 2 in \cref{even NCP} follows from the following observations.
\begin{enumerate}
    \item When $d$ is of type (I$\pm$), we have $r_j= r_j'\pm 2$ if $2x_j= 2i$ and $r_j= r_j'$ otherwise. Hence, each $r(L_i)$ changes by an even number, namely $0$, $2$, or $-2$.
    \item When $d$ has type (N), we have that $2i - 2$ is a part of $\lambda'$ and $2i$ is a part of $\lambda$. In particular, $\lambda$ is obtained from $\lambda'$ by replacing the part $2i - 2 \in \lambda'$ with $2i$. Moreover, by the construction in \cref{clusters def}, the two numbers $2i$ and $2i - 2$ are always sent to the same open cluster. Therefore, the numbers $r(L_i)$ do not change.
   
\end{enumerate}

%are the same unless $2x_j= 2i$ when $d$ has type (I+) or (I-). When $2x_j=2i$. We have $r_j-r_{j'}= 2, -2$ and $0$ when $d$ has type (I+), (I-) and (N). Therefore, the numbers $r(L_i)$ stay even.

%Think about the proof and if I want to write anything about non-special orbit.

\section{A fixed-point variety and $\Stab(e)$}
In this section, we finish the proof of \cref{main thm} by showing $S(e)\subset \Stab(e)$ for $e\in \fsp_{2n}$. We first recall some basic Lie theory and introduce a fixed-point subvariety of $\spr$ that helps to determine the set $\Stab(e)$. 

From $\lambda\in \cP_C(2n)$, we obtain two subpartitions $\lambda^0$ and $\lambda^1$ that consist of the even and odd parts of $\lambda$, respectively. In the notation of Section 2, $\lambda^0$ can be written as $((2x_i)^{r_i})_{1\leqslant i\leqslant \ell}$ where the superscripts denote the multiplicities of the parts. Since $\lambda\in \cP_C(2n)$, the partition $\lambda^1$ takes the form $((2y_i+1)^{2q_i})_{1\leqslant i\leqslant \ell'}$, where $\ell'$ is the number of distinct odd parts in $\lambda$. 

From the Jacobson-Morozov theorem (see, e.g., \cite[Section 3.7]{CG}), there exists an $\fsl_2$-triple in $\fsp_{2n}$ containing $e$. Consider such a triple $(e,h,f)$, and let $Q_e$ be its stabilizer in $Sp_{2n}$. Let $T_h$ be the one-dimensional torus in $Sp_{2n}$ with Lie algebra $\Lie(T_h) = \bC h$. Then $Q_e$ and $T_h$ naturally act on $\spr$. 

From \cite[Section 5]{collingwood1993nilpotent}, the group $Q_e$ is isomorphic to a product $\prod_{i=1}^{\ell} O_{r_i}\times \prod_{i=1}^{\ell'} Sp_{2q_i}$. We may view $A_e\cong (\cyclic{2})^{\oplus \ell}$ as the component group of $Q_e$. Moreover, we can choose a subgroup of $\prod_{i=1}^{\ell} O_{r_i}$ that projects isomorphically onto $A_e$. Thus we can regard $A_e$ as a subgroup of $\prod_{i=1}^{\ell} O_{r_i}\subset Q_e$. The action of $A_e$ on $\Irr(\spr)$ is independent of the choice of the subgroup. 

Let $\fix$ denote the fixed-point variety $\spr^{T_h}$. This variety is smooth and projective (see, e.g., \cite[3.6 (a)]{DLP}). We write $\Con(\fix)$ for the set of connected components of $\fix$. Since the actions of $A_e$ (viewed as a subgroup of $Q_e$) and $T_h$ commute, $A_e$ acts on $\fix$ and therefore on $\Con(\fix)$. 

Consider the standard representation $V$ of $\fsp_{2n}$. The torus $T_h$ acts on $V$, yielding a mod-$2$ weight decomposition $V=V_0\oplus V_1$. Both subspaces $V_0$ and $V_1$ are stable under the action of $e, h$ and $f$ (viewed as an endomorphism of $V$). The restriction of $e$ to $V_0$ and $V_1$ gives us two nilpotent elements $e_0$ and $e_1$, with partitions $\lambda^0, \lambda^{1}$, respectively. 

The restrictions of the triple ($e$, $f$, $h$) to $V_0$ and $V_1$ yield two $\fsl_2$-triples associated with $e_0$ and $e_1$. The groups $Q_{e_0}$ and $Q_{e_1}$ then naturally identify with the first and the second factors in the isomorphism $Q_e\cong \prod_{i=1}^{\ell} O_{r_i}\times \prod_{i=1}^{\ell'} Sp_{2q_i}$. Since $Q_{e_1}$ is connected, the natural injection $Q_{e_0}\hookrightarrow Q_{e}$ induces a bijection between $A_e$ and $A_{e_0}$. Therefore, it makes sense to view $\Irr(\cB_{e_0})$ as a $A_e$-set.

The nilpotent element $e_0 \in \fsp(V_0)$ is even, and therefore Richardson. We cite an important result from \cite{DLP} which relates the sets $\Irr(\cB_{e_0})$ and $\Con(\cB_{e_0}^{gr})$.
\begin{pro}\cite[Corollary 3.3]{DLP} \label{from dlp}
    Consider $X \in \Con(\cB_{e_0}^{gr})$, the closure of its attracting locus in $\cB_{e_0}$ is a single irreducible component of $\cB_{e_0}$. Therefore, $\Con(\cB_{e_0}^{gr}) \cong \Irr(\cB_{e_0}^{gr})$ as $A_e$-sets.
\end{pro}
Next, we state two main results of this section.
\begin{pro}\label{reduce to even}
    There is a (noncanonical) injection of $A_e$-sets $$\Irr(\cB_{e_0})\hookrightarrow \Irr(\spr).$$
\end{pro}

\begin{pro} \label{enough stab}
    Assume that the partition $\lambda$ of $e$ is of the form $((2x_i)^{r_i})_{1\leqslant i\leqslant \ell}$ with $x_1>\dots > x_\ell$. For any subgroup $A\in S(e)$, there exists a component $X_A\in \Con(\fix)$ whose stabilizer is $A$.
\end{pro}
We first explain how these propositions imply \cref{main thm}; their proofs are given afterwards.
\begin{proof}[Proof of \cref{main thm}]
   Recall from \cref{from dlp} that the sets $\Con(\cB_{e_0}^{gr})$ and $\Irr(\cB_{e_0}^{gr})$ are isomorphic as $A_{e_0}$-sets. Hence, by \cref{enough stab}, we have $S(e_0) \subset \Stab(e_0)$. Next, \cref{reduce to even} implies $S(e_0) \subset \Stab(e_0) \subset \Stab(e)$. Finally, since the description of $S(e)$ in \cref{main thm} does not depend on the odd parts of $\lambda$, we see that $S(e) = S(e_0) \subset \Stab(e) \subset S(e)$, where the last containment follows from \cref{main combi}. Therefore, we have $S(e)= S(e_0)= \Stab(e)$.

\end{proof}

\begin{proof}[Proof of \cref{reduce to even}]
    The proof consists of two steps.
    \begin{enumerate}
        \item \textbf{Step 1:} We construct an injection $\iota$ from $\Irr(\cB_{e_0}) = \Con(\cB_{e_0}^{gr})$ to $\Con(\spr)$.
        \item \textbf{Step 2:} We identify the image of this injection as a subset of $\Irr(\spr)$ via the closures of corresponding attracting loci.
    \end{enumerate}

    \textbf{Step 1}. We begin by discussing some properties of the connected components of $\fix$. An element of $\fix$ is an isotropic flag $U^\bullet= U^1\subset \dots \subset U^{2n}$ of $V$ that is stable under $e$ and $T_h$. Each such flag $U^\bullet\subset \fix$ determines a unique tuple of $T_h$-weights $w(U^\bullet) =(w_1,\dots,w_{2n})$ such that the $T_h$-weights of $U^i$ are $w_1,\dots,w_i$ as a multiset. Given a tuple of $2n$-weights $\alpha$, define the subvariety $$X_\alpha:= \{U^\bullet\subset \fix, w(U^\bullet)=\alpha\}.$$ This $X_\alpha$ is either empty or an $A_e$-orbit of $\Con(\fix)$. For a tuple $\alpha$, let $\alpha^0$ and $\alpha^1$ be the subtuples consisting of the even and odd weights, respectively. Then $X_{\alpha^0}$ and $X_{\alpha^1}$ can be viewed as subvarieties of $\cB_{e_0}^{gr}$ and $\cB_{e_1}^{gr}$, respectively. Moreover, we have an isomorphism $X_\alpha \cong X_{\alpha^0} \times X_{\alpha^1}$ (\cite[Section 9.1]{hoang2024geometryfixedpointsloci}).

    Recall that we write $\lambda^0 = ((2x_i)^{r_i}){0 \leqslant i \leqslant \ell}$ and $\lambda^1 = ((2y_i+1)^{2q_i}){0 \leqslant i \leqslant \ell'}$. Let $n_0 = \sum_{i=1}^{\ell} x_i r_i$ and $n_1 = \sum_{i=1}^{\ell'} (2y_i + 1) q_i$, so that $n = n_0 + n_1$. In what follows, we consider a distinguished tuple of weight $\alpha^1$ so that the geometry of $X_{\alpha^1}$ is nice.

    For $i \geqslant 0$, write $\beta_i$ for the tuple of weights $(2i, 2i - 2, \dots, -2i)$, and let $(\beta_i)^l$ denote the concatenation of $l$ copies of $\beta_i$. Define $\beta$ as the concatenated tuple $((\beta_{2y_1})^{q_1}, \dots, (\beta_{2y_{\ell'}})^{q_{\ell'}})$. Let $\beta^r$ be the reverse of $\beta$, obtained by listing its entries from right to left. Write $-\beta^r$ for the tuple obtained by changing the sign of each weight in $\beta^r$. Let $\alpha^1$ be $(\beta, -\beta^r)$.

    From \cite[Section 9.3]{hoang2024geometryfixedpointsloci}, the variety $X_{(\beta, -\beta^r)}$ is a product of towers of projective bundles, and hence it is a connected component of $\cB_{e_1}^{gr}$. Now, given $\alpha^0$ such that $X_{\alpha^0} \subset \cB_{e_0}^{gr}$ is nonempty, define the tuple $\alpha = (\beta, \alpha^0, -\beta^r)$. Then we have an isomorphism $$X_{(\beta,\alpha^0, -\beta^{r})}\cong X_{(\beta,-\beta^{r})}\times X_{\alpha^0}.$$
    
    It follows that there is a bijection between the $A_e$-sets $\Con(X_{\alpha^0})$ and $\Con(X_{(\beta, \alpha^0, -\beta^r)})$. As a result, we obtain an injection of $A_e$-sets
    $$\iota: \Con(\cB_{e_0}^{gr})\hookrightarrow \Con(\fix)$$
    whose image consists of the connected components of those varieties
    $$\{X_\alpha \subset \fix, \alpha \text{ starts with } \beta \text{ and ends with } -\beta^r\}.$$

    \textbf{Step 2.} We claim that if $\alpha$ begins with $\beta$ and ends with $-\beta^r$, then the attracting locus of the $T_h$-action of any connected component of $X_\alpha$ has maximal dimension, equal to $\dim \spr$. This implies that $\iota(\Con(\cB_{e_0}^{gr}))$ naturally identifies with a subset of $\Irr(\spr)$, and the proposition follows.

    We recall a dimension formula from \cite[Proposition 3.2]{DLP}. The action of $T_h$ on $\fg = \fsp_{2n}$ induces a weight decomposition $\fg = \bigoplus \fg(i)$. Let $\fp := \bigoplus_{i \geqslant 0} \fg(i)$ be the corresponding parabolic subalgebra, and let $P$ be the associated parabolic subgroup. The group $P$ acts on the flag variety of $G$ with finitely many orbits $\cO$. For each such orbit, define $\cB_{e,\cO} := \spr \cap \cO$. Each nonempty variety $\cB_{e,\cO}$ is the attracting locus of a component of $\fix$ (\cite[Section 3.4]{DLP}). The dimension of $\cB_{e,\cO}$ is computed as follows. Let $\fb$ (a Borel subalgebra of $\fg$) be a point in $\cB_{e,\cO}$. Then:
    
    $$\dim \cB_{e,\cO}= \dim (\fp/\fp\cap \fb)- \dim (\fg_{\geqslant 2}/\fg_{\geqslant 2}\cap \fb),$$ 
    in which $\fg_{\geqslant 2}= \bigoplus_{i\geqslant 2} \fg(i)$. Hence, for $\fb \in \fix$ (i.e., a Borel subalgebra stable under the $T_h$-action), the dimension of the corresponding attracting locus is given by $$\dim \cB_{e,\cO}= \dim \cB(\fg(0))+\dim \fg(1)- \dim \fb\cap \fg(1)$$     
    where $\cB(\fg(0))$ denotes the flag variety of $\fg(0)$. This quantity depends only on $\dim (\fb \cap \fg(1))$.

    For $\fb \in X_\alpha$, the dimension of the intersection $\dim \fb\cap \fg(1)$ can be computed in terms of the weight tuple $\alpha = (\alpha_1, \dots, \alpha_{2n})$. Specifically, $\dim (\fb \cap \fg(1))$ is equal to half the number of pairs $(i < j)$ such that $\alpha_i - \alpha_j = 1$. Let $E(\alpha)$ denote the number of such pairs. The proof concludes with the calculation in \cref{correct dimension}.
\end{proof}
In the following lemma, we use the notation $\alpha^0$, $\beta$ as in the previous proof. Recall that we write $V$ for the standard representation of $\fg$ and we have a mod-$2$ weight decomposition $V=V_0\oplus V_1$.
\begin{lem}\label{correct dimension}
    For any $\alpha^0$ consisting of $2n_0$ weights of $V_0$, we have the following.
    \begin{equation} \label{dim equation}
        \dim \spr- \dim \cB(\fg(0))= \frac{1}{2} \dim \fg(1)= \dim \fg(1)- \frac{1}{2}E(\beta, \alpha, -\beta^r)
    \end{equation}
\end{lem}
\begin{proof}
    Let $V=\oplus V(i)$ be the weight space decomposition of the $T_h$-action. Then we have $$\dim \fg(1)=\sum_{i\geqslant 0}\dim V(i)\dim V(i+1).$$ Next, consider the number $E(\beta, \alpha^0,-\beta^r)$. Each weight $2t$ in $\beta$ contributes $\dim V({2t-1})$ (this is the multiplicity of $2t-1$ in $V_0$) to $E(\beta, \alpha^0,-\beta^r)$. Similarly, the contribution of a weight $2t$ in $-\beta^r$ is $\dim V(2t+1)$. The tuple $\beta$ consists of $\frac{1}{2}\dim V(2t)$ weights $2t$, so we get 
    $$E(\beta, \alpha^0,-\beta^r)= \frac{1}{2}(\sum_{i\in \bZ} \dim V(i)\dim V(i+1))= \sum_{i\geqslant 0}\dim V(i)\dim V(i+1)$$
    because $\dim V(i)= \dim V(-i)$. Therefore, we get 
    $$ \dim \fg(1)- \frac{1}{2}E(\beta, \alpha, -\beta^r)= \frac{1}{2}(\sum_{i\geqslant 0}\dim V(i)\dim V(i+1))= \frac{1}{2}\dim \fg(1).$$
    This proves the second equality in (\ref{dim equation}). We proceed to prove the first equality in (\ref{dim equation}). 
    
    Recall that we write $e_0$ and $e_1$ for the restriction of $e$ to $V_0$ and $V_1$. Consider $\fg_0= \fsp(V_0)$ and $\fg_1= \fsp(V_1)$ as subalgebras of $\fg= \fsp(V_0\oplus V_1)$. Write $\fg^e, \fg_0^{e_0}$ and $\fg_1^{e_1}$ for the centralizers of $e, e_0$ and $e_1$ in $\fg, \fg_0$ and $\fg_1$, respectively. The two algebras $\fg_0$ and $\fg_1$ are stable under the action of $T_h$, so it makes sense to write $\fg_0(0)$ and $\fg_1(0)$ for the corresponding eigenspaces. We have $\fg(0)= \fg_0(0)\oplus \fg_1(0)$. By \cite[Corollary 3.3]{DLP}, the dimensions of the Springer fibers $\cB_{e_0}$ and $\cB_{e_1}$ are the dimensions of the flag varieties of $\fg_0(0)$ and $\fg_1(0)$. Hence $\dim \spr- \dim \cB(\fg(0))$ $= \dim \spr -\dim \cB_{e_0}-\dim \cB_{e_1}$.

    By \cite[Corollary 3.3.24]{CG}, and \cite[Corollary 6.1.4]{collingwood1993nilpotent}, we have $\dim \spr= \frac{1}{2}(\dim \fg^e- n)$, where $n$ is the rank of $\fg$. The same formula applies to $e_0\in \fg_0$ and $e_1\in \fg_1$, so we have $\dim \spr -\dim \cB_{e_0}-\dim \cB_{e_1}$ $=\frac{1}{2}(\dim \fg^e- \dim \fg_{0}^{e_0}- \dim \fg_{1}^{e_1})$. To compute this number, we recall a formula in the proof of \cite[Theorem 6.1.3]{collingwood1993nilpotent}. This formula expresses the dimension of $\dim \fg^e$ in terms of the partition $\lambda$. Write $p_i$ for the multiplicity of the part $i$ in $\lambda$, we then have
    $$\dim \fg^e= p_1(\frac{1}{2}p_1+\frac{1}{2}+p_2+p_3+\dots)+ 2p_2(\frac{1}{2}p_2+ p_3+ p_4+\dots)+ 3p_3(\frac{1}{2}p_3+\frac{1}{2}+p_4+p_6+\dots)+\dots$$
    
    Note that the same formula applies for $\fg_{0}^{e_0}$ and $\fg_{1}^{e_1}$ and the two partitions $\lambda^0$ and $\lambda^1$. Hence, we have 
    $$\dim \fg^e- \dim \fg_{0}^{e_0}- \dim \fg_{1}^{e_1}= p_1(p_2+p_4+\dots)+ 2p_2(p_3+p_5+\dots)+3p_3(p_4+p_5+\dots)+\dots$$
    $$=(p_1+p_3+\dots)(p_2+p_4+\dots)+ (p_2+p_4+\dots)(p_3+p_5+\dots)+\dots = \dim V(0)\dim V(1)+ \dim V(1)\dim V(2)+\dots$$
     In conclusion, we have obtained
     $$\dim \spr- \dim \cB(\fg(0))= \frac{1}{2}(\dim \fg^e- \dim \fg_{0}^{e_0}- \dim \fg_{1}^{e_1})= \frac{1}{2}(\sum_{i\geqslant 0}\dim V(i)\dim V(i+1))= \frac{1}{2}\dim \fg(1).$$
\end{proof}

We conclude this section with the proof of \cref{enough stab}.
\begin{proof}[Proof of \cref{enough stab}]
    This proposition is a corollary of \cite[Lemma 9.17]{hoang2024geometryfixedpointsloci}, we refer the reader there for more details. Here we briefly explain the relevant content of \cite[Lemma 9.17]{hoang2024geometryfixedpointsloci} and how \cref{enough stab} follows. 
     
    Let $l$ be the number of parts of $\lambda$, so $l= \sum_{i=1}^\ell r_i$. Let the standard form of $\lambda$ be $(\lambda_1\geqslant \dots\geqslant \lambda_l)$. Consider an abstract vector space $V\cong \bC^l$ with a basis $v_1,\dots,v_l$ and the dot product. For a subset $J$ of $\{1,\dots,\l\}$, we write $V_J$ for the subspace $\Span(v_j)_{j\in J}$. On $V$, we have a natural action of $$\prod_{i=1}^{\ell} O_{r_i}= \prod_{i=1}^{\ell} O (V_{[1+r_1+\dots+r_{i-1}, r_1+\dots+r_i]}).$$ The two key points of \cite[Lemma 9.17]{hoang2024geometryfixedpointsloci} are the following:
    \begin{enumerate}
        \item Consider a chain of intervals with even lengths $$C= \{[a_1,b_1]\subsetneq\dots \subsetneq [a_p,b_p]\subset [1,l]\}$$
        where each interval $[a_j,b_j]$ is of the form $[1+r_1+\dots+r_{i_1}, r_1+\dots+r_{i_2}]$ for some $1\leqslant i_1< i_2\leqslant \ell$.
        Consider the variety $X_C$ that parameterizes the partial isotropic flags $$\{U^1\subset\dots.\subset U^p\subset V, U^i\subset V_{[a_i,b_i]}, \dim U^i= \frac{1}{2}(b_i-a_i+1)\}.$$ Then there exists a tuple $\alpha$ so that $X_\alpha \subset \fix$ is isomorphic to a tower of projective bundles over $X_C$.  
        \item For disjoint chains $C_1,..,C_v$ of such intervals, there exists a tuple $\alpha$ so that $X_\alpha \subset \fix$ is isomorphic to a tower of projective bundles over $X_{C_1}\times\dots \times X_{C_v}$.
    \end{enumerate}

    Now recall that each subgroup $A_\cL \in S(e)$ is defined via an even NCP $\cL = \{L_0, L_1, \dots, L_k\}$ of $\{2x_1, \dots, 2x_\ell\}$. By \cref{end point}, an NCP is determined by the numbers $\min(L_i)$ and $\max(L_i)$ for $1 \leqslant i \leqslant k$.

    Let $m_i = \sum_{j=1}^{\min(L_i)} r_j$ and $m_i' = \sum_{j=1}^{\max(L_i)} r_j$. Then the intervals $J_i := [m_{i-1} + 1, m_i'] \subset [1, l]$ for $1 \leqslant i \leqslant k$ satisfy the condition that for any two distinct intervals, they are either disjoint or one contains the other. Moreover, each $J_i$ has even length by the assumption that $\cL$ is an even NCP (\cref{even NCP}).

    Applying the two key points above to the set of intervals $\{J_i\}$, we obtain a variety $X_{A'} \subset \fix$. Choosing a connected component $X_A$ of $X_{A'}$, we obtain the desired component. The verification that this orbit has stabilizer $A_\cL$ follows from the fact that each maximal orthogonal Grassmannian of $V_{[m_{i-1}+1, m_i']}$ has two connected components, which are permuted by the component group of $O(V_{[m_{i-1}+1, m_i']})$.
\end{proof}
%fix the mess with  the 0-cluster
\section{Applications and conjectures}
%maybe there is no proof to write, if Lusztig already defined it
%left cell and constructible repn
%Lustig's quotient, Lusztig's generator. From Sommer's paper, Lusztig's quotient is generated by $t_l$ so that\dots Describe the map explicitly. No need to bring  repn of Weyl group in. The statement is, under the quotient map, the set S(e) is sent to the set that comes from left cell modules. Discuss the map as explicitly as possible. Maybe write a remark about TL  pattern. Then discuss other things about domino tableau (so we need left cell here). State the best possible result. NCP seems to fits his non intersecting pattern. Or TL? Get the formula then choose what to do.

%Write the bijection first
%Find other interpretation on internet
%Then write Section 4
%Decide what to do with NCP stuff later+ fix things related to the $0$-cluster.

%say something about special element?
In this section, we consider only special nilpotent elements $e$. To a special nilpotent orbit, Lusztig has assigned a quotient $\bar{A}(e)$ of $A_e$, known as Lusztig's canonical quotient, along with a subset $c_e \subset W$ called a two-sided cell. The two-sided cells $c_e$ are further partitioned into finitely many left cells. Each left cell gives rise to a left cell representation (constructible representation) of $W$. These representations are introduced in \cite{Lusztig1979} and are classified by certain subgroups $H_\sigma$ of $\bar{A}(e)$ (see, e.g., \cite{Lusztig1982}, \cite[Section 6]{Losev_2014}). 
\subsection{S(e) and Lusztig' subgroups}
Consider a partition $\lambda = (2\lambda_1\geqslant \dots \geqslant 2\lambda_l)$ of $2n$ and a corresponding nilpotent element $e\in \fsp_{2n}$. Let $W$ be the Weyl group of $\fsp_{2n}$. The orbit $G.e$ is a special nilpotent orbit in $\fsp_{2n}$ (see, e.g., \cite[Section 7]{collingwood1993nilpotent}). In what follows, we briefly recall combinatorial descriptions of $\bar{A}(e)$ and subgroups $H_\sigma$. Our exposition follows \cite[Section 6]{Losev_2014} and \cite[Section 5]{Sommers2001}.

%From Springer theory, Weyl group $W$ on $H^{top}(\spr)$ (see, e.g ). This action is multiplicity free. Consider the irreducible representations of $W$ that appear in both $H^{top}(\spr)$ and some left cell reprsentations of $c_e$. Taking their direct sum, we get $Spr(\bO) \subset H^{top}(\spr)$. Let $A'\subset A_e$ be the subgroup that stabilizes $Spr(\bO)$. The Lusztig quotient $\bar{A}(e)$ is then $A_e/A'$.

Recall that $A_e\cong (\cyclic{2})^{\oplus \ell}$ is generated by $t_1,\dots,t_\ell$ where $\lambda = ((2x_i)^{r_i})_{1\leqslant i\leqslant \ell}$ and $x_1>\dots>x_\ell$. Let $K_e$ be the kernel of the map $A_e\twoheadrightarrow \bar{A}(e)$.

\begin{pro}[\cite{Sommers2001}, Theorem 6] \label{Lquotient}
The subgroup $K_e \subset A_e$ is generated by the elements $t_m t_{m+1}$ for those indices $1 \leqslant m \leqslant \ell$ such that $r_1 + \dots + r_m$ is odd, where we set $t_{\ell+1} = 1$. In particular, $\bar{A}(e) \cong A_e$ if all $r_i$ are even.
\end{pro}
As a consequence, the quotient $\bar{A}(e)$ is generated by the images of those $t_m$ for which $r_1 + \dots + r_m$ is even. Let $1 \leqslant m_1 < \dots < m_e \leqslant \ell$ be the indices satisfying this condition. For notational simplicity, we will write $t_{m_i}$ for the image of $t_{m_i}$ in $\bar{A}(e)$.

We now describe the subgroups $H_\sigma$ in terms of these generators. These subgroups are in bijection with a set of certain Temperley–Lieb (TL) patterns (see, e.g., \cite[Section 6.5]{Losev_2014}). In particular, consider the set $$Z=\{1, 2{m_i}, 2{m_i}+1\}_{1\leqslant i\leqslant e}$$ consisting of $2e + 1$ points. The following definition is rephrased from \cite[Section 6.5]{Losev_2014}.
\begin{defin} \label{TL}
    A TL pattern $T_\sigma$ of $Z$ is a partition of $Z$ into $e+1$ subsets $M_0,\dots,M_e$ that satisfy the following conditions:
\begin{enumerate}
    \item $M_0$ consists of a single odd number $d$. Each $M_i$ consists of two elements $a_i <b_i$ such that $(a_i-d)(b_i-d)> 0$.
    \item For any $i \neq j$, the intervals $[a_i, b_i]$ and $[a_j, b_j]$ do not intersect, or one is contained in the other.
\end{enumerate}
\end{defin}
Given such a $T_\sigma$, the corresponding subgroup $H_\sigma\subset \bar{A}(e)$ is generated by the following elements
\begin{itemize}
    \item $t_{m_{i_1}}t_{m_{i_2}}$ where $\{2m_{i_1}, 2m_{i_2-1}+1\}$ is a subset in $T_\sigma$ for some $i_1<i_2\leqslant e$,
    \item and $t_{m_j}$ if $T_\sigma$ contains the set $\{2m_j, 2m_e+1\}$.
\end{itemize}
  In the notation of \cite[Section 6.5]{Losev_2014}, their generators for $\bar{A}(e)$ correspond to our elements $t_{m_1} t_{m_2},\ t_{m_2} t_{m_3},$ $ \dots,\ t_{m_{e-1}} t_{m_e}$ and $t_{m_e}$. Our description of $H_\sigma$ follows from \cite[Section 6.5]{Losev_2014} by a simple change of variables. The following remark explains why $H_\sigma$ can be determined from subsets of the form ${2m_{i_1},\ 2m_{i_2 - 1} + 1}$ in $T_\sigma$. %\footnote{The author thanks Eric Sommers for pointing out this change of generators.}. 
\begin{rem}\label{TL end point}
    Condition 2 in \cref{TL} implies that each subset in a TL pattern $T_\sigma$ always consists of one even and one odd element. As a result, a TL pattern $T_\sigma$ of the set $Z = \{1, 2m_1, \dots, 2m_e + 1\}$ is uniquely determined by either of the following:
    \begin{itemize}
        \item All subsets of the form $\{2a, 2b + 1\}$ in $T_\sigma$, or
        \item All subsets of the form $\{2a + 1, 2b\}$ in $T_\sigma$.
    \end{itemize}
    For example, suppose that we know all subsets of the form $\{2a, 2b + 1\}$ in $T_\sigma$. Let $Z'$ be the union of these subsets. If $Z' = \emptyset$, define $A = 2m_e + 2$; otherwise, let $A$ be the smallest element in $Z'$. Then $2A - 1$ must form its own singleton subset in $T_\sigma$. The remaining subsets are of the form $[c, d]$, where $c$ is even and $d$ is the smallest element in $Z \setminus Z'$ such that $d > c$.
\end{rem}

\begin{exa}
    Consider $\lambda = (100^3, 38^3, 16^2)$. Then $A_e$ is generated by $t_1, t_2, t_3$. The subgroup $K_e$ is generated by $t_1 t_2$, so $\bar{A}(e)$ is generated by $t_2$ and $t_3$, corresponding to the parts $38$ and $16$ in $\lambda$. The set $Z$ is $\{1, 4, 5, 6, 7\}$. We list the TL patterns $T_\sigma$ and their corresponding subgroups $H_\sigma$:
    \begin{enumerate}
        \item $T_\sigma = \{\{1\}, \{6,7\}, \{4,5\}\}$: $H_\sigma$ is generated by $t_3$ (from $\{6,7\}$) and $t_2 t_3$ (from $\{4,5\}$), so $H_\sigma = \bar{A}(e)$.
        \item $T_\sigma = \{\{1\}, \{4,7\}, \{5,6\}\}$: $H_\sigma$ is generated by $t_2$ (from $\{4,7\}$).
        \item $T_\sigma = \{\{5\}, \{1,4\}, \{6,7\}\}$: $H_\sigma$ is generated by $t_3$ (from $\{6,7\}$).
        \item $T_\sigma = \{\{7\}, \{1,4\}, \{5,6\}\}$: $H_\sigma$ is trivial.
        \item $T_\sigma = \{\{7\}, \{1,6\}, \{4,5\}\}$: $H_\sigma$ is generated by $t_2$ (from $\{4,5\}$).
    \end{enumerate}    
\end{exa}
Recall from Section 2 that we defined a set $S(e)$ consisting of stabilizers of the $A_e$-action on $\Irr(\spr)$. We have the following theorem.
\begin{thm} \label{evidence}
    Consider the quotient map $p: A_e\rightarrow \bar{A}(e)$. The images of the subgroups in $S(e)$ are precisely the subgroups $H_\sigma\subset \bar{A}(e)$ defined by the Temperley-Lieb patterns in \cref{TL}.
\end{thm}
\begin{proof}
    The proof consists of two steps.
    
    \textbf{Step 1: Reduction to the case }$\bar{A}(e)= A_e$. Consider an even NCP $\cL = {L_0, L_1, \dots, L_k}$ of $B_\lambda = {0, 2x_1, \dots, 2x_\ell}$ and the corresponding group $A_\cL$ described in Section 2.1. We now describe the image $p(A_\cL)$ in terms of this NCP. Recall that we write $m_1, \dots, m_e$ for the indices $m$ such that $r_1 + \dots + r_m$ is even. From $\cL$, we obtain an NCP $\bar{\cL}$ of the subset ${0, 2x_{m_1}, \dots, 2x_{m_e}}$ as follows.
    
    Consider the equivalence relation $\sim_p$ on $\{0,2x_{1},\dots,2x_{m_e}\}$ by $2x_i\sim_p 2x_j$ if there is some $1\leqslant l\leqslant e$ so that $m_{l-1}+1\leqslant i\leqslant j\leqslant m_l$. Write $M$ for the set $\{2x_i\}_{i>m_e}$. Then the sets $(L_i\setminus M)/\sim_p$ (some of them may be empty) give an NCP $\bar{\cL}$ of $(B_\lambda\setminus M)/\sim_p\cong \{0, 2x_{m_1},\dots, 2x_{m_e}\}$. 
    
    From $\bar{\cL}$, we obtain a group $A_{\bar{\cL}}$ in terms of $t_{m_i}$. From \cref{Lquotient}, regarding these $t_{m_i}$ as the generators of $\bar{A}(e)$, we may view $A_{\bar{\cL}}$ as a subgroup of $\bar{A}(e)$. From the construction, it follows that $p(A_\cL)= A_{\bar{\cL}}$.
    
   Consider an NCP $\bar{\cL} = \{\bar{L}_1, \dots, \bar{L}_k\}$ of the set $\{0, 2x_{m_1}, \dots, 2x_{m_e}\}$. By \cref{end point}, we can reconstruct an NCP $\cL$ of $B_\lambda$ using the data of $\min(\bar{L}_i)$ and $\max(\bar{L}_i)$ for each $i$. The resulting partition $\cL$ is even, due to the assumption that each partial sum $r_1 + \dots + r_{m_i}$ is even. Applying the process described in the previous paragraph to this $\cL$ recovers $\bar{\cL}$. Therefore, the image of $p$ is parametrized by the NCPs of $\{0, 2x_{m_1}, \dots, 2x_{m_e}\}$. In other words, to prove the proposition, it suffices to consider the case when $A_e \cong \bar{A}(e)$.

    \textbf{Step 2: A bijection between NCPs and TL patterns}. Assume that $r_1,\dots,r_\ell$ are even, or equivalently, $A_e \cong \bar{A}(e)$. Now $m_i= i$ for $1\leqslant i\leqslant \ell$. The set $S(e)$ is in bijection with the set of NCPs $\cL = {L_0, \dots, L_k}$ of ${0, 2x_1, \dots, 2x_\ell}$. From each TL pattern of $Z = {1, 2, \dots, 2\ell + 1}$, defined by $\ell + 1$ sets $M_0, \dots, M_\ell$, we construct an NCP by the following steps:
    \begin{enumerate}
        \item Identify the sets $M_{i_1},\dots,M_{i_d}$ of the form $\{2a_j+1<2b_j\}$ for some $0\leqslant a_j< b_j\leqslant \ell$ and $1\leqslant j\leqslant d$.
        \item Let $I_j$ be the interval $[a_j+1,b_j]$. Define $L_j=\{2x_l, l\in I_j\setminus (\bigcup_{I_{j'}\subsetneq I_j} I_{j'})\}$ for $1 \leqslant j \leqslant d$. The set $L_0$ is the complement of these sets in ${0, 2x_1, \dots, 2x_\ell}$.
    \end{enumerate}
    In summary, a set $M_i$ of the form ${2a + 1 < 2b}$ corresponds to a set $L_j$ in the NCP with $min(L_j) = 2x_{a + 1}$ and $max(L_j) = 2x_b$. From \cref{end point}, an NCP is uniquely determined by the numbers $min(L_j)$ and $max(L_j)$ for $j \neq 0$. On the other hand, given an NCP $L_0, \dots, L_k$, we use the numbers $min(L_j)$ and $max(L_j)$ to obtain all sets of the form ${2a_j + 1 < 2b_j}$ and recover a TL pattern (see \cref{TL end point}). Therefore, we obtain a bijection between the NCPs of ${0, 2x_1, \dots, 2x_\ell}$ and the TL patterns of ${1, 2, \dots, 2\ell + 1}$. We identify the corresponding subgroups of $A_e$ to complete the proof.

    Given $T_\sigma$, the corresponding subgroup $H_\sigma \subset \bar{A}(e)$ is generated by $t_{i_1}t_{i_2}$ for each pair ${2i_1, 2i_2 - 1}$ that forms a subset in $T_\sigma$ for some $i_1 < i_2 \leqslant e$, and by $t_j$ if $T_\sigma$ contains the set $\{2j, 2\ell + 1\}$. Equivalently, we can descirbe $H_\sigma$ using the sets of forms ${2a+1 < 2b}$ in $T_\sigma$ to see how it matches the subgroup $A_\cL$. For a set ${2i_1, 2i_2 - 1}$, there are two possibilities. 
    \begin{enumerate}
        \item There exists a set of the form $\{2a+1<2b\}$ in $T_\sigma$ so that $[2i_1,2i_2-1]\subset [2a+1,2b]$. In this case, the element $t_{i_1}t_{i_2}$ correspond to the first type of generators of $A_\cL$ in Section 2.1.
        \item If there is no such $a,b$, we argue that both $t_{i_1}$ and $t_{i_2}$ are in $H_\sigma$. We first look for the index $l$ so that $\{2l, 2\ell+1\}$ appears in $T_\sigma$, then we have $t_l\in H_\sigma$. Next, we look at the elements $l'$ so that $\{2l', 2l-1\}$ apears in $T_\sigma$, then we have $t_{l'}t_l \in H_\sigma$. This implies $t_{l'}\in H_\sigma$. In summary, we get that $H_\sigma$ contains all $t_l$ such that $l \notin \bigcup [a+1, b]$, where $\{2a+1 < 2b\}$ ranges over the subsets in $T_\sigma$. These elements are precisely the second type of generators of $A_\cL$.
    \end{enumerate}
%This latter description matches exactly the subgroup corresponding to the NCP constructed from $T_\sigma$ (see the discussion preceding \cref{main thm}).

\end{proof}

\begin{cor}
    Consider the case where $\lambda$ takes the form $(2\lambda_1,  2\lambda_1, 2\lambda_2, 2\lambda_2,\dots, 2\lambda_d, 2\lambda_d)$. We have $A_e= \bar{A}(e)$ by \cref{Lquotient}. Then the set $S(e)$ is precisely the groups $H_\sigma$ defined by Lusztig. This result proves a conjecture in \cite[Section 3]{lusztig2024constructible}. In fact, we have obtained a stronger result (in \cite{lusztig2024constructible}, the conjecture is stated for $\lambda_i= i$).
\end{cor}
%And can I relate to the paper superspecial
%do we have the subgroups H_\sigma when e is not special
\subsection{A conjecture}
Recall that we write $K_e$ for the kernel of the surjection $A_e \twoheadrightarrow \bar{A}(e)$. The group $K_e$ acts on $\Irr(\spr)$, and we now introduce a notion that helps to describe the set of $K_e$-orbits in $\Irr(\spr)$.

\begin{defin}\label{good clusters}[Good clusters]
Consider $\lambda$ of the form $((2x_i)^{r_i}){1\leqslant i\leqslant \ell}$. Let $T^\pm$ be a signed domino tableau with underlying shape $\lambda$. We have the surjective map $b_T: B_\lambda \rightarrow \cO\cC_T$. We say that an open cluster $\cC$ of $T^\pm$ is \textit{good} if the preimage $b_T^{-1}(\cC)$ satisfies the following condition:
\begin{itemize}
\item Consider the indices $m$ such that $r_1 + \dots + r_m$ is odd. Then $x_m \in b_T^{-1}(\cC)$ if and only if $x_{m+1} \in b_T^{-1}(\cC)$.
\end{itemize}
\end{defin}
In particular, if all $r_i$ are even, then all the open clusters of $T^\pm$ are good. Let $\Sigma DT_{op, cl}^{g}(\lambda)$ be the subset of $\Sigma DT_{op, cl}(\lambda)$ that consists of all the domino tableaux having signs $+$ on the open clusters that are not good. From the description of $K_e$ in \cref{Lquotient}, we have a bijection between two $\bar{A}(e)$ sets
$$\Irr(\spr)/K_e= \{K_e \text{-orbits of Irr}(\spr)\}\leftrightarrow \Sigma DT_{op, cl}^{g}(\lambda).$$

We then have a commutative diagram in which the fibers of the vertical maps are $\bar{A}(e)$-orbits.
\begin{center}
\begin{tikzcd}
    Irr(\spr)/K_e \arrow{r}{\sim} \arrow{d}{} \& \Sigma DT_{op, cl}^{g}(\lambda) \arrow{d}{} \\    
    Irr(\spr)/A_e \arrow{r}{\sim}  \& \Sigma DT_{cl}(\lambda)
\end{tikzcd}
\end{center}

From \cite[Lemma 4.2]{Pietraho2004} there is a bijection between $\Sigma DT_{cl}(\lambda)$ and the set $SDT(\lambda)$ that parameterizes the standard domino tableaux of shape $\lambda$. In \cite{Gar}, the author proves that the set $SDT(\lambda)$ also parameterizes the left cells in $c_e$, the two-sided cell that corresponds to $e$. Consider the $\bar{A}(e)$-set $Y= \bigsqcup (\bar{A}(e)/H_\sigma)$ where $\sigma$ runs over all left cells in $c_e$. From \cite{McGovern1996}, the set $Y$ admits a combinatorial description in terms of domino tableaux (not necessarily) of shape $\lambda$. Therefore, we propose the following combinatorial conjecture.
\begin{conj} \label{multiplicity}
    There is a bijection between two $\bar{A}(e)$-sets $\Sigma DT_{op, cl}^{g}(\lambda)$ and $Y$ so that the following diagram commutes.
    \begin{center}
\begin{tikzcd}
    Irr(\spr)/K_e \arrow{r}{\sim} \arrow{d}{} \& \Sigma DT_{op, cl}^{g}(\lambda) \arrow{d}{} \arrow{r}{\sim} \& Y= \bigsqcup (\bar{A}(e)/H_\sigma) \arrow{d} \\    
    Irr(\spr)/A_e \arrow{r}{\sim} \& \Sigma DT_{cl}(\lambda) \arrow{r}{\sim} \& \{SDT(\lambda)\} = \{\sigma, \text{left cells in }c_e\}
\end{tikzcd}
\end{center}
\end{conj}
When $\lambda$ is of the form $((2x_i)^{r_i})_{1\leqslant i\leqslant \ell}$ for $r_i$ even (so $K_e$ is trivial), \cref{multiplicity} is much stronger than \cref{main thm}. In particular, \cref{multiplicity} not only gives a list of stabilizers of the $A_e$-action on $\Irr(\spr)$ but also computes the multiplicities of the corresponding orbits. On the other hand, when $K_e$ is not trivial, it is not yet clear how to obtain the set $S(e)$ from the subgroups $H_\sigma$ defined by Lusztig.
  
\section{Cases where $\fg$ of exceptional types}
In this section, we present an algorithm to describe the $A_e$-set $\Irr(\spr)$ for exceptional $\fg$ using the Springer correspondence. The proof of \cref{main conj} will then follow naturally from our description of $\Irr(\spr)$ and the construction of the set $Y$ given in \cite{Losev_2014}.

Assume that $G$ is of adjoint type. In these cases, the group $A_e$ is isomorphic to a symmetric group $S_i$ for some $i \leqslant 5$. We recall some notation from Section 3. Let $(e, h, f) \subset \fg$ be an $\fsl_2$-triple. The adjoint action of $h$ gives a grading $\fg = \bigoplus \fg(i)$. Let $G_0$ be the subgroup of $G$ with Lie algebra $\fg(0)$. Then $G_0$ acts on $\fg(2)$ with a dense orbit containing $e$. Let $B_0$ be a Borel subgroup of $G_0$.

Let $T_h \subset G$ be the one-dimensional torus whose Lie algebra is $\mathbb{C}h$, and let $\fix$ denote the fixed-point variety $\spr^{T_h}$. We have an injection of $A_e$-sets $\Irr(\spr) \hookrightarrow \Con(\fix)$ that sends an irreducible component $X$ of $\spr$ to the connected component $Y$ of $\fix$ whose attracting locus is dense in $X$.

Let $S'(e)$ denote the set of conjugacy classes of stabilizers of the $A_e$-action on $\Con(\fix)$. Then we have the inclusion $\Stab(e) \subset S'(e)$. This containment is sufficient to determine the $A_e$-set $\Irr(\spr)$ for exceptional $\fg$. The details are given below.

We first explain some results related to $\Con(\fix)$ and the set $S'(e)$. By \cite[Sections 2.13, 3.7]{DLP}, the set $\Con(\fix)$ can be understood via the set of $B_0$-stable subspaces of $\fg(2)$ that intersect the open $G_0$-orbit $G_0.e$ nontrivially. In particular, consider such a subspace $U_\alpha \subset \fg(2)$, and let $M_\alpha$ be the subgroup of $G_0$ consisting of elements $g \in G_0$ such that $g^{-1} e g \in U_\alpha$. Then the variety $X_\alpha = M_\alpha / B_0$ is isomorphic to an $A_e$-orbit of connected components of $\fix$.

The subspaces $U_\alpha$ have been studied in detail in \cite{DLP} and \cite{Sommers2006}. We cite the following result from the diagrams at the end of \cite{Sommers2006}. For the labels of nilpotent orbits, we use the notation in \cite[Section 13]{carter1993finite}.
\begin{pro}\label{stab exceptional}\leavevmode
    \begin{enumerate}
        \item When $\bar{A}(e)=S_2$, we have $S'(e)\subset \{S_2, \text{the trivial group}\}$.
        \item When $\bar{A}(e)=S_3$, the set $S'(e)$ consists of $S_3, S_2$ and the trivial group, except for the following two cases.
        \begin{itemize}
            \item When $e$ is of type $G_2(a_1)$, $S'(e)$ consists of $S_2$ and $S_3$. 
            \item When $e$ is of type $E_8(b_6)$, $S'(e)$ consists of $S_2$ and $\bZ/3\bZ$ and $S_3$.     
        \end{itemize}
        \item When $\bar{A}(e)=S_4$, the set $S'(e)$ consists of $S_4, S_3, S_2, S_2\times S_2$ and $D_4$ where the last group is the dihedral group of order $8$. 
        \item When $\bar{A}(e)=S_5$, the set $S'(e)$ consists of $S_5, S_4, S_3, S_2, S_3\times S_2, S_2\times S_2$ and $D_4$.
\end{enumerate}
In particular, the cardinality of $\Stab(e)\subset S'(e)$ is less than or equal to the number of conjugacy classes of $A_e$.
\end{pro}
For $x\in A_e$, write $\chi(x)$ for the character value of $x$ in the Springer representation $H^{top}(\spr)$. We have the following theorem.
\begin{thm}\label{main exceptional}
   The $A_e$-set $\Irr(\spr)$ is uniquely recovered from the numerical invariants $\chi(x), x\in A_e$.  
\end{thm}
\begin{proof}
    Each conjugacy class of $A_e$ determines a numerical value $\chi(x)$. This value $\chi(x)$ can be expressed in terms of the multiplicities of the $A_e$-orbits in $\Irr(\spr)$. Consequently, we obtain a system of linear equations, where the variables represent these multiplicities. Since the number of elements in $\Stab(e)$ is less than or equal to the number of conjugacy classes of $A_e$, the proof of \cref{main exceptional} reduces to verifying that certain square matrices are invertible. These matrices are derived from the character tables of the permutation representations $K_0(A_e / A)$, for subgroups $A \subset A_e$ with $A \in S'(e)$. The detailed calculations are given below.
    
    Since we are considering permutation representations $K_0(A_e / A)$ of $A_e$, the value $\chi(x)$ is equal to the number of $A$-cosets fixed by the adjoint action of $x$. The following tables display the corresponding character values. The first columns list the $A_e$-orbits in $\Irr(\spr)$, while the first rows list representatives of the conjugacy classes of $A_e$. We use the notation $(i_1,\dots,i_j)$ to denote a cycle in the symmetric group, and write $()$ for the conjugacy class of the identity element.

    Case $A_e= S_2$. 
    \begin{center}
\begin{tabular}{ |c|c|c| } 
 \hline
     & () & (12) \\ 
 \hline
 $S_2$ & 2 & 0 \\ 
 \hline
 $S_2/S_2$ & 1 & 1 \\ 
 \hline
\end{tabular}
\end{center}

Case $A_e= S_3$. In this case, we have two tables.
    \begin{center}
\begin{tabular}{ |c|c|c|c|c| } 
 \hline
     & () & (12) & (123) \\ 
 \hline
 $S_3$ & 6 & 0& 0 \\ 
 \hline
 $S_3/S_2$ & 3 & 1& 0 \\ 
 \hline
 $S_3/S_3$ & 1 & 1 &1\\ 
 \hline
\end{tabular}
\end{center}

\begin{center}
\begin{tabular}{ |c|c|c|c|c| } 
 \hline
     & () & (12) & (123) \\ 
 \hline
 $S_3/C_3$ & 2 & 0& 2 \\ 
 \hline
 $S_3/S_2$ & 3 & 1& 0 \\ 
 \hline
 $S_3/S_3$ & 1 & 1 &1\\ 
 \hline
\end{tabular}
\end{center}

Case $A_e= S_4$. To count the number of fixed points, we view $S_4/S_3$, $S_4/D_4$, $S_4/(S_2\times S_2)$ and $S_4/S_2$ as the set of $\{i\}$, $\{\{i,j\},\{k,l\}\}$, $\{i,j\}$, and $(i,j)$ for $i,j,k,l$ pairwise distinct in $\{1,2,3,4\}$.
    \begin{center}
\begin{tabular}{ |c|c|c|c|c|c|c| } 
 \hline
     & () & (12) & (123) & (12)(34) & (1234) \\ 
 \hline
 $S_4/S_4$ & 1 & 1 & 1 & 1 & 1 \\ 
 \hline
 $S_4/S_3$ & 4 & 2 & 1 & 0 & 0 \\ 
 \hline
 $S_4/D_4$ & 3 & 1 & 0 & 1 & 0 \\ 
 \hline
 $S_4/(S_2\times S_2)$ &6  & 2 & 0 & 3 & 0 \\ 
 \hline
 $S_4/S_2$ & 12 & 2 & 0 & 0 & 0 \\ 
 \hline
\end{tabular}
\end{center}

Case $A_e= S_5$. To count the number of fixed points, we view $S_5/S_4$, $S_5/(S_3\times S_2)$, $S_5/D_4$, $S_5/S_3$, $S_5/S_2\times S_2$, and $S_5/S_2$ as the set of $\{i\}$, $\{i,j,k\}$, $\{i, \{j,k\}, \{l,h\}\}$, $(i,j)$, $\{(i,\{i,j,k\})\}$, and $(i,j,k)$  for $i,j,k,l,h$ pairwise distinct in $\{1,2,3,4,5\}$.
    \begin{center}
\begin{tabular}{ |c|c|c|c|c|c|c|c| } 
 \hline
     & () & (12) & (123) & (12)(34) & (1234)& (12) (345)& (12345) \\ 
 \hline
 $S_5/S_5$ & 1 & 1 & 1 & 1 & 1 & 1 & 1 \\ 
 \hline
 $S_5/S_4$ & 5 & 3 & 2 & 1 & 1& 0& 0 \\ 
 \hline
 $S_5/(S_3\times S_2)$ & 10 & 4 & 1 & 2 & 0& 1& 0 \\ 
 \hline
 $S_5/D_4$ &15  & 3 & 0 & 3 & 1 & 0 & 0 \\ 
 \hline
 $S_5/S_3$ &20  & 6 & 2 & 0 & 0 & 0 & 0 \\ 
 \hline
 $S_5/(S_2\times S_2)$ &30  & 6 & 0 & 2 & 0 & 0 & 0 \\ 
 \hline
 $S_5/S_2$ &60  & 6 & 0 & 0 & 0 & 0 & 0 \\ 
 \hline
\end{tabular}
\end{center}
It is straightforward to check that the matrices that come from these tables are invertible. We give an example of how the multiplicities of the orbits $A_e/A$ can be computed from $\chi(x)$ for $x\in A_e$. The other cases are similar.

    Assume that $\bar{A}(e)=S_3$ and $\Stab(e)\subset \{\{1\}, S_2, S_3\}$. Write $a_1$, $a_2$, and $a_3$ for the multiplicities of the orbits $S_3$, $S_3/S_2$, and $S_3/S_3$, respectively. From the table for $S_3$ in the proof of \cref{main exceptional}, we have the following system of equations. In this example, we use $1$ for the identity element in $S_3$.
    \begin{equation*}
    \begin{cases}
      \chi(1)= 6a_1+ 3a_2+ a_3\\
      \chi((12))= 3a_1+ a_2\\
      \chi((123))= a_3
    \end{cases}\,
\end{equation*}
Therefore, we have $a_3= \chi((123))$, $a_2= \chi(1)- 2\chi((12))-\chi((123))$, and $a_1= \chi((12))- \frac{\chi(1)-\chi((123))}{3}$.
\end{proof}

\begin{rem}
    \cref{main exceptional} gives us an algorithm to compute the set $\Irr(\spr)$. A natural question is whether we have $\Stab(e)= S'(e)$ as in Section 3 for classical $\fg$. It turns out that there are cases where $\Stab(e)$ is a proper subset of $S'(e)$. In particular, this happens for precisely $3$ orbits in types $E_7$ and $E_8$. These orbits are referred to as exceptional in \cite{Losev_2014}. %want to determine which orbits?
\end{rem}
Another consequence of \cref{main exceptional} is the following.

\begin{pro}
    \cref{main conj} is true for $\fg$ of exceptional type. 
\end{pro}
\begin{proof}
    Recall that in the setting of \cref{main conj}, $e$ is a special nilpotent element. There is nothing to prove when $\bar{A}(e)$ is trivial. We now consider the cases where $\bar{A}(e)$ is nontrivial. These cases are listed below (see, e.g., \cite[Sections 13.1–13.2]{carter1993finite}, \cite[Section 6]{Losev_2014}).\footnote{The exposition in \cite[Section 6]{Losev_2014} contains some minor inaccuracies. For example, when $e$ is of type $E_8(b_6)$, the group $\bar{A}(e)$ is $S_2$, not $S_3$.} %The author thanks Eric Sommers for pointing this out.}

\begin{enumerate}
\item $\bar{A}(e) = S_2$.
\item $\bar{A}(e) = S_3$. This occurs for eight special orbits: one in type $G_2$, one in $E_6$, two in $E_7$, and four in $E_8$.
\item $\bar{A}(e) = S_4$ for a single orbit in type $F_4$.
\item $\bar{A}(e) = S_5$ for a single orbit in type $E_8$.
\end{enumerate}
Note that $A_e= \bar{A}(e)$ whenever $\bar{A}(e)$ is not trivial and $e$ is not of type $E_8(b_6)$ (see, e.g. \cite[Section 13]{carter1993finite}). When $e$ is of type $E_8(b_6)$, we have $A_e= S_3$ and $\bar{A}(e)= S_2$. From \cref{main exceptional}, the $\bar{A}(e)$-set $\Irr(\spr)/K_e$ is determined by the character value $\chi(x)$ of the module $K_0(\Irr(\spr)/K_e)$, $x\in \bar{A}(e)$. 

Recall that the set $Y$ in \cref{main conj} is the disjoint union $\bigsqcup \bar{A}(e)/H_\sigma$ where $\sigma$ runs over the left cells of the two-sided cell corresponding to $e$. We can check the containment $H_\sigma\in S'(e)$ using the tables in \cite{carter1993finite} or the results in \cite[Section 6]{Losev_2014}. Hence, it makes sense to use the algorithm in the proof of \cref{main exceptional} to compute the multiplicities of the orbits $\bar{A}(e)/A$ in $Y$. 

One of the main results of \cite{Losev_2014} is that the set $Y$ parameterizes the finite dimensional irrreducible modules with regular integral character of the W-algebra attached to $e$. From \cite[Theorem 7.4]{Losev_2014}, we have an isomorphism of $\bar{A}(e)$-modules $K_0(Y)\cong H^{top}(\spr)^{K_e}= K_0(\Irr(\spr)/K_e)$. Hence, the character values $\chi(x)$ are the same for the two modules $K_0(Y)$ and $K_0(\Irr(\spr)/K_e)$. Because we compute $Y$ and $\Irr(\spr)/K_e$ using the same algorithm and the same numerical invariants, the two sets are isomorphic as $\bar{A}(e)$-sets.
\end{proof}

\printbibliography
\end{document}